\documentclass{article} 
\usepackage{pinlabel}
\newcommand{\pic}[2]{\raisebox{-.5\height}{\includegraphics[scale=#2]{#1}}}
\def\closure{\pic{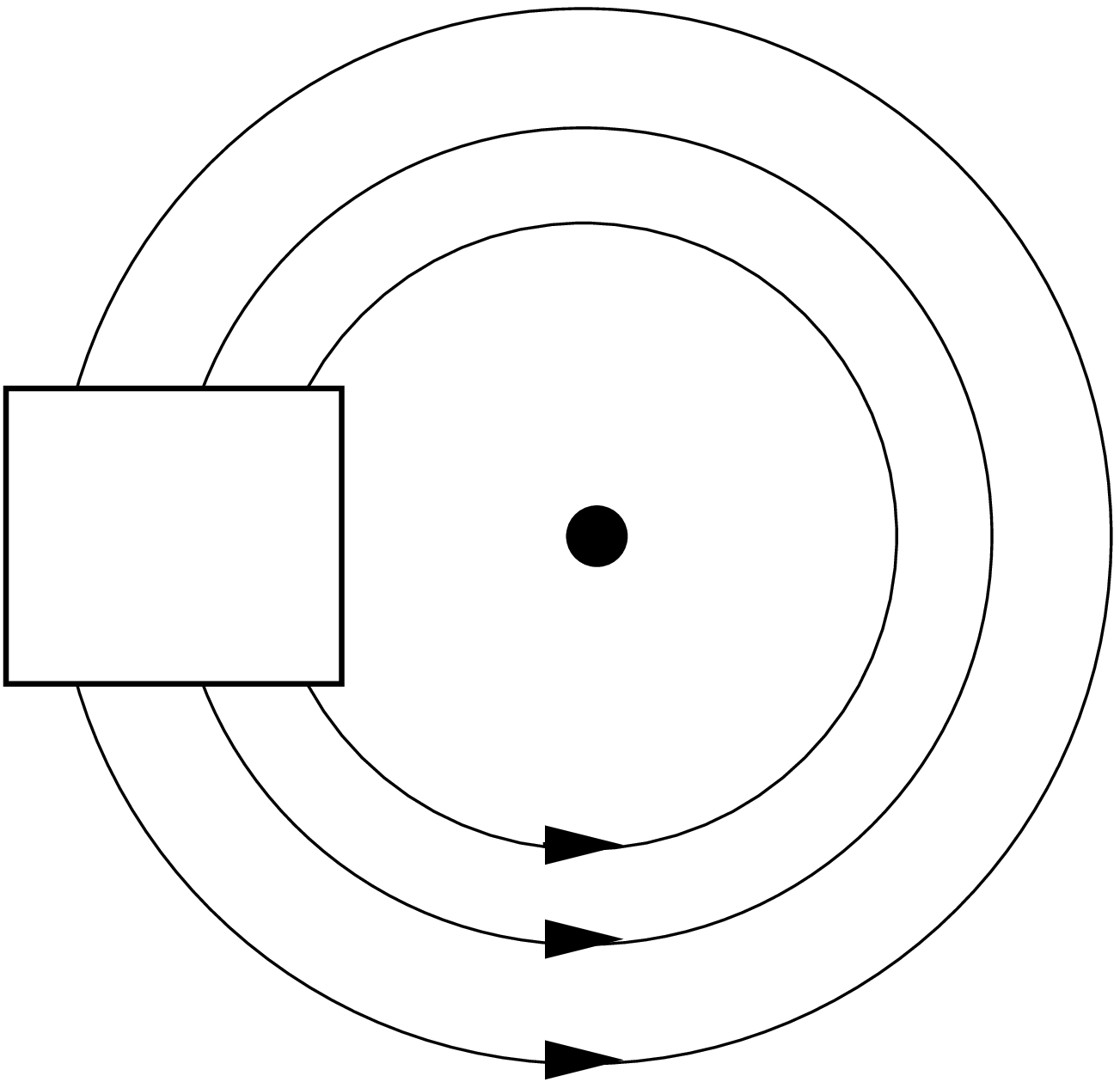}{.250}}
\def\Am{\pic{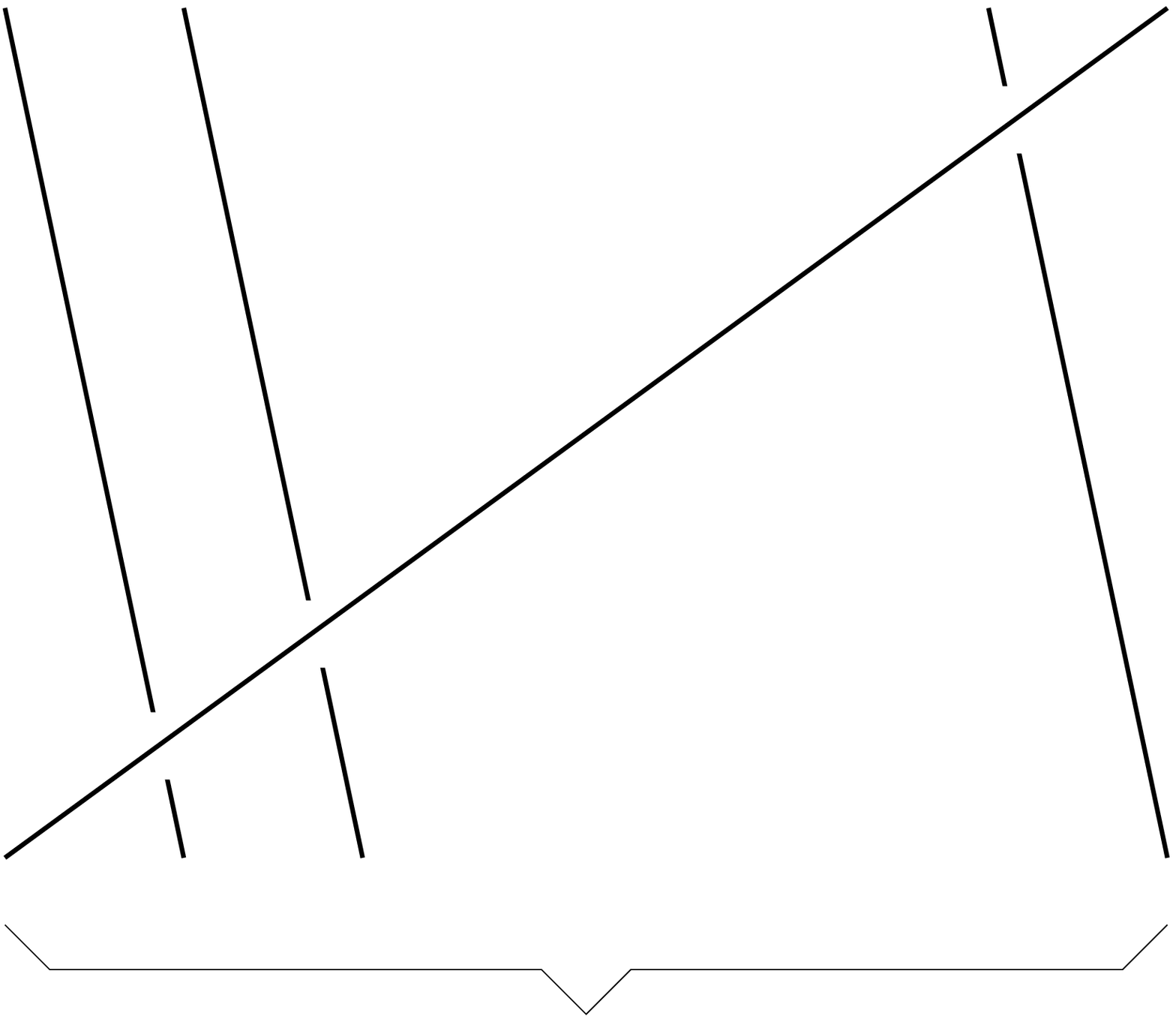}{.150}}
\def\Ambarra{\pic{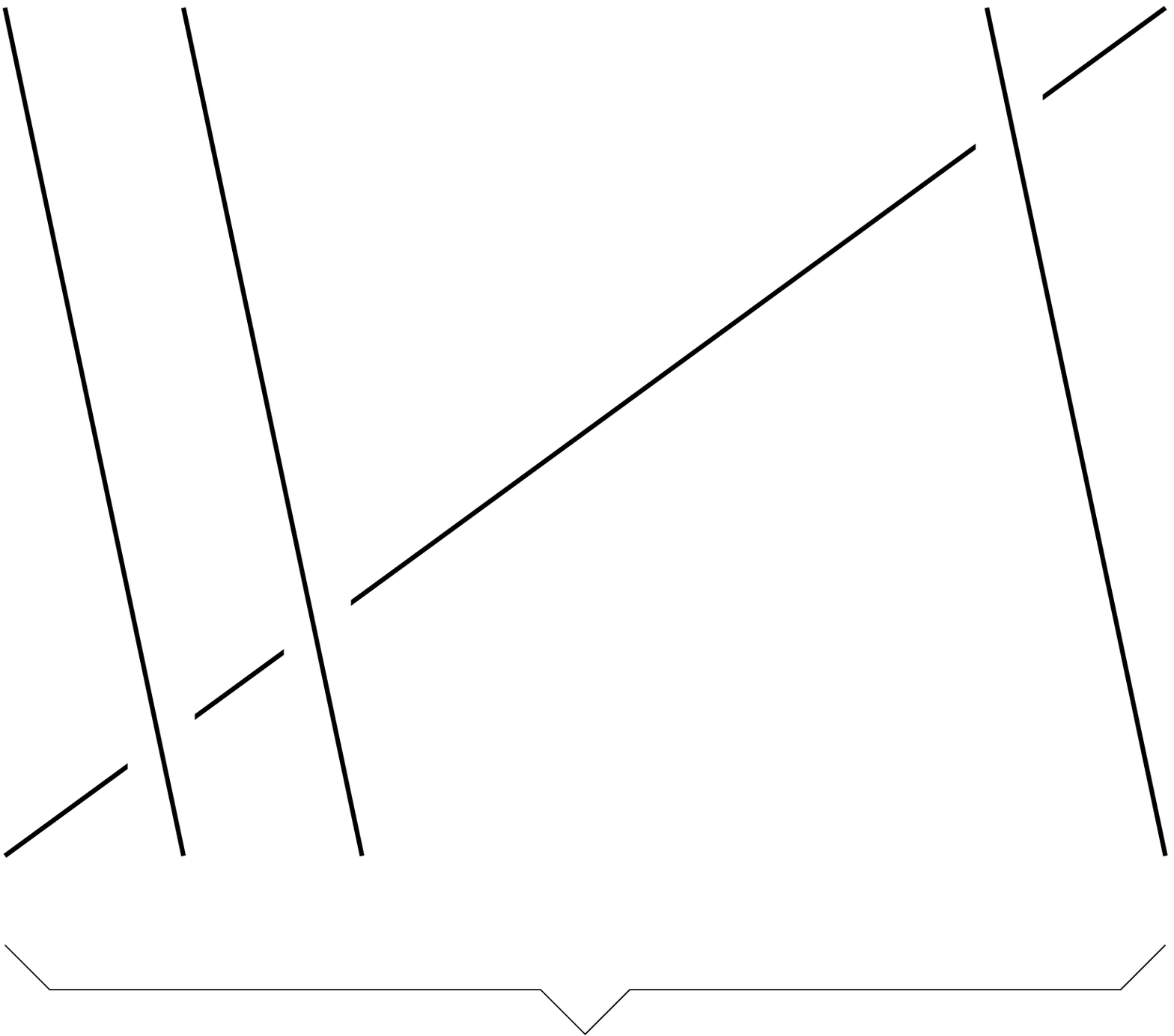}{.150}}
\def\Aij{\pic{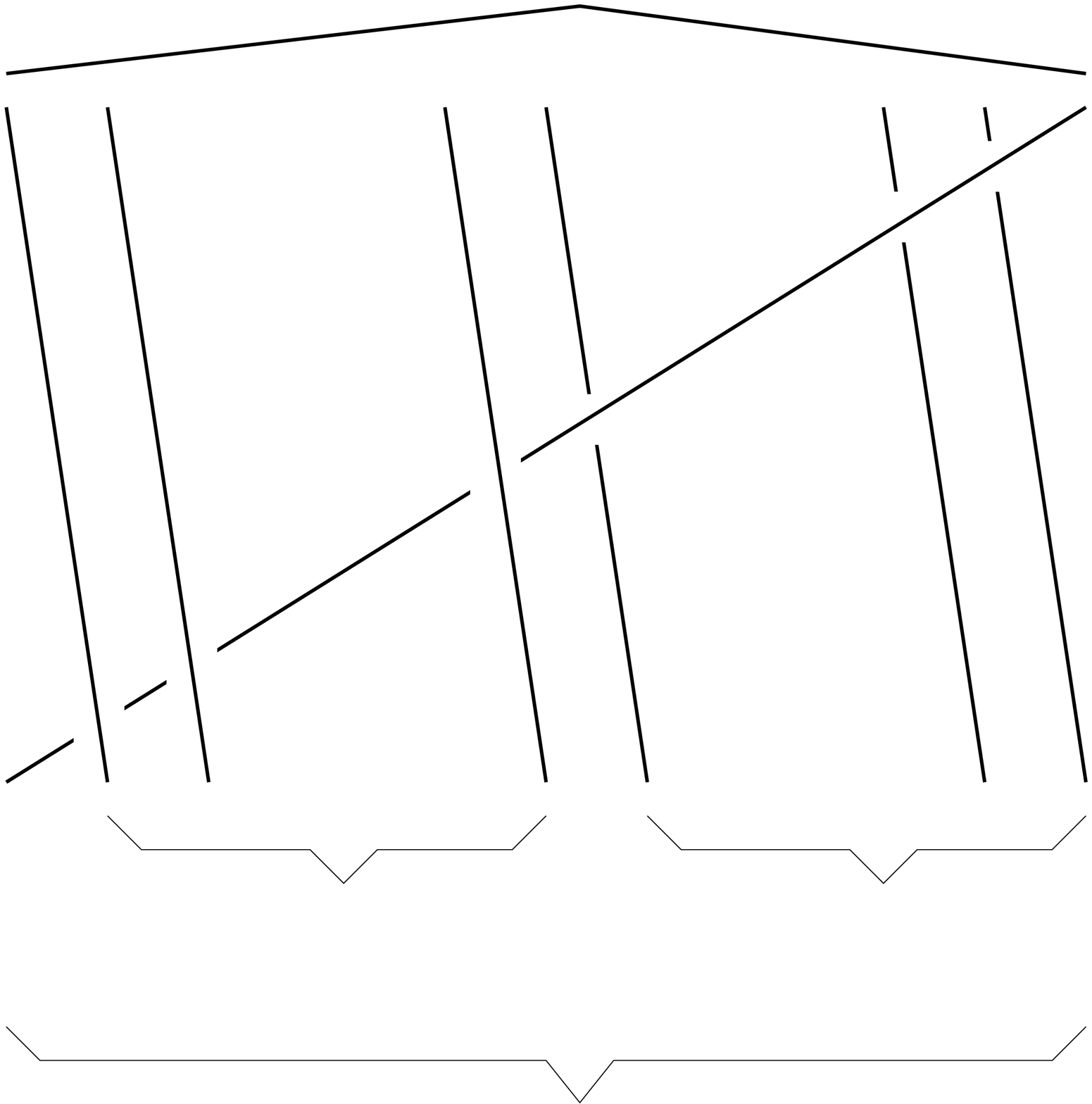}{.150}}
\def\AijMarcado{\pic{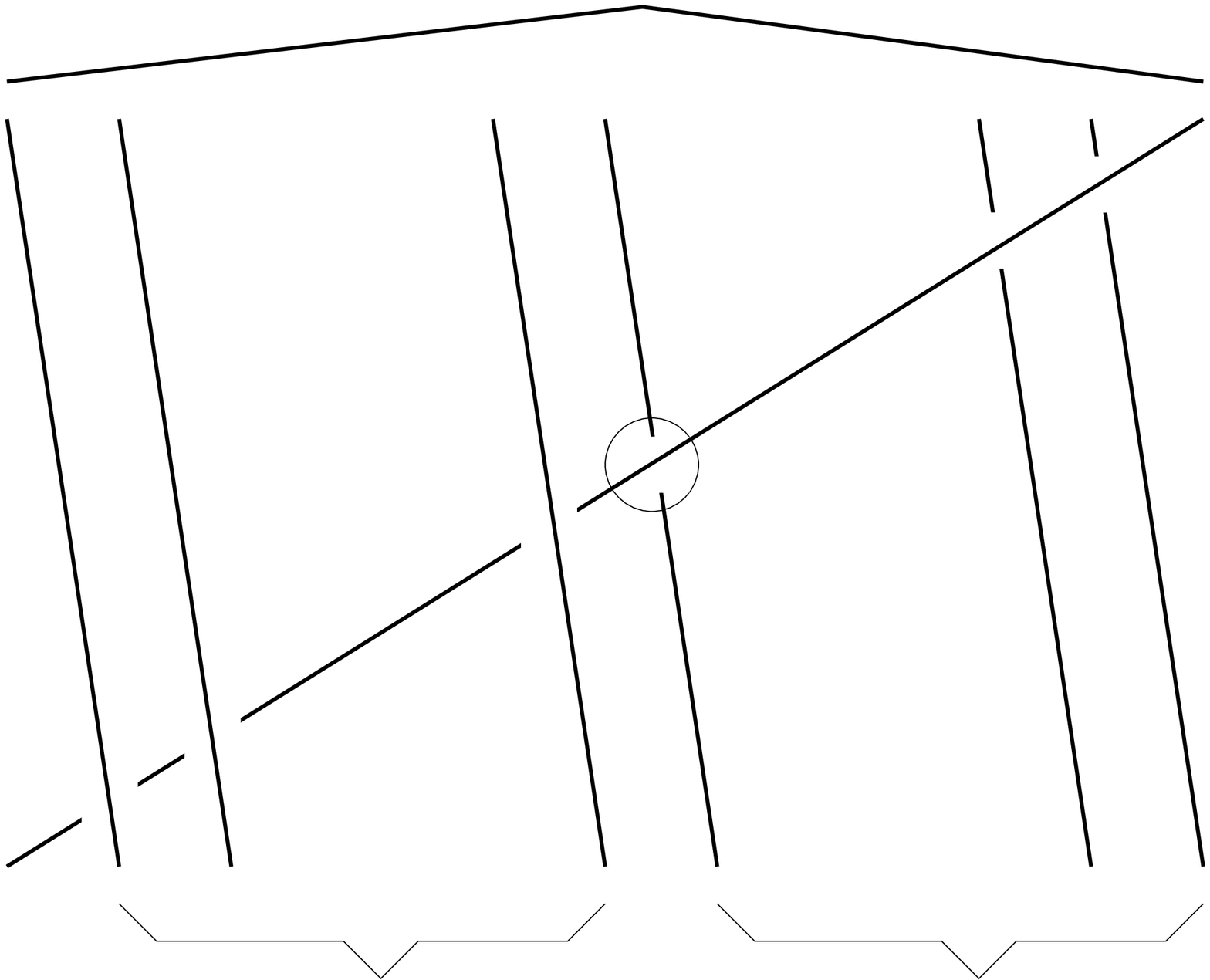}{.150}}
\def\meridianmap{\pic{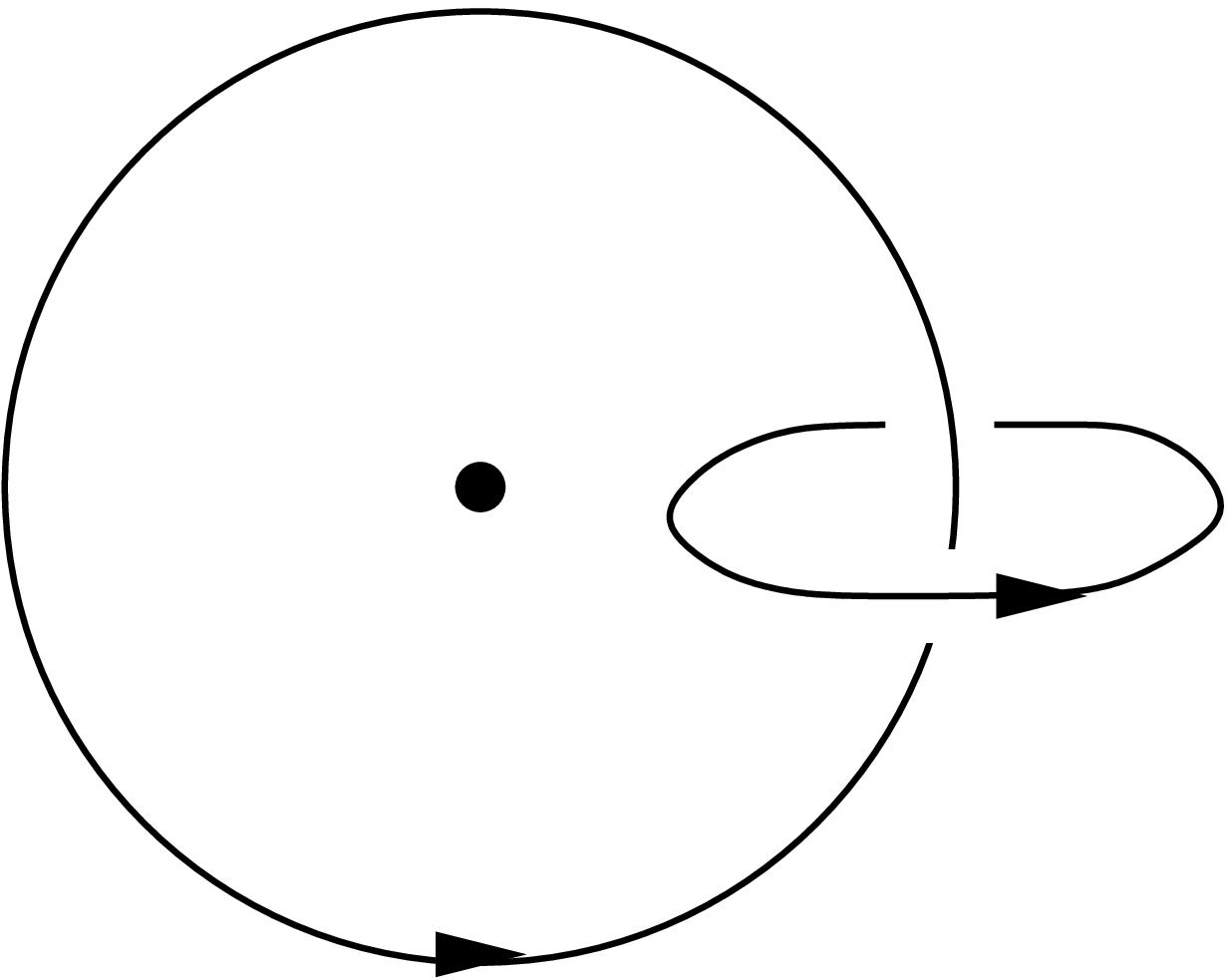}{.20}}
\def\meridianmapRev{\pic{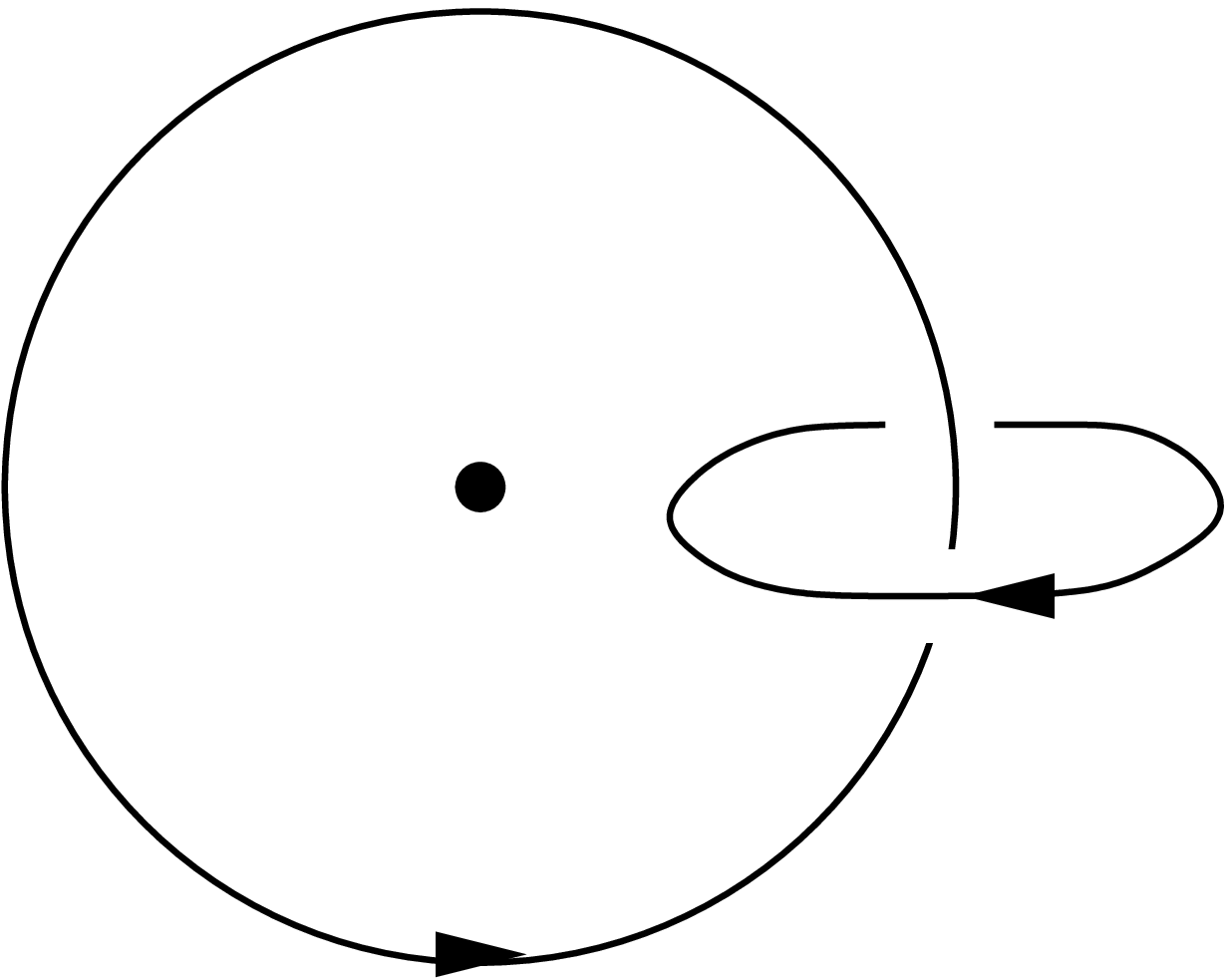}{.20}} 
\def\Xm{\pic{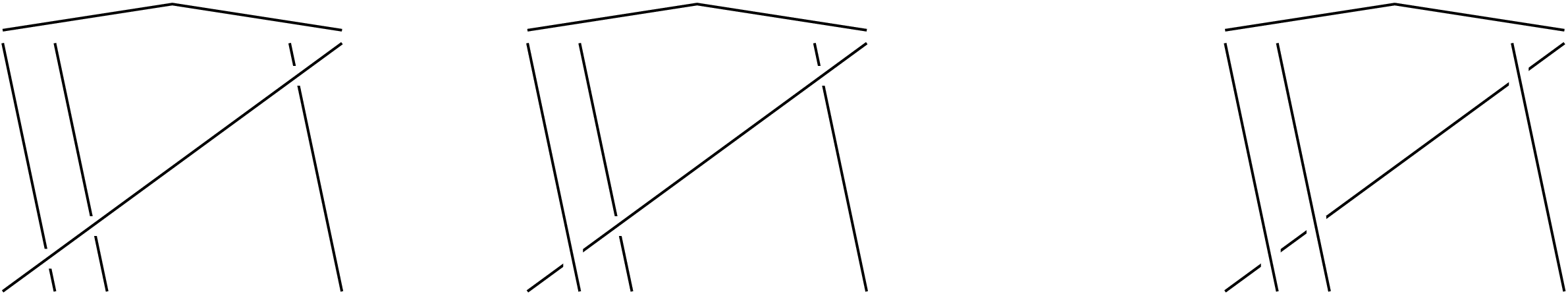}{.250}}
\def\Heckemeridian{\pic{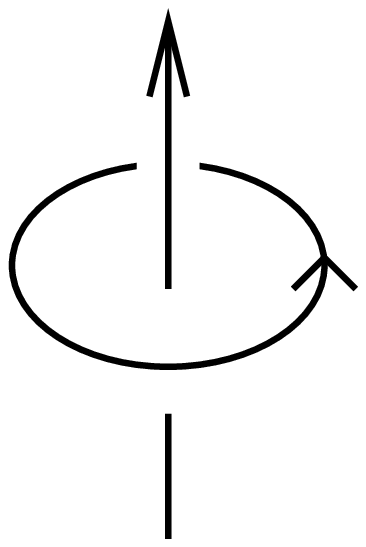} {.300}}
\def\deltaA{\pic{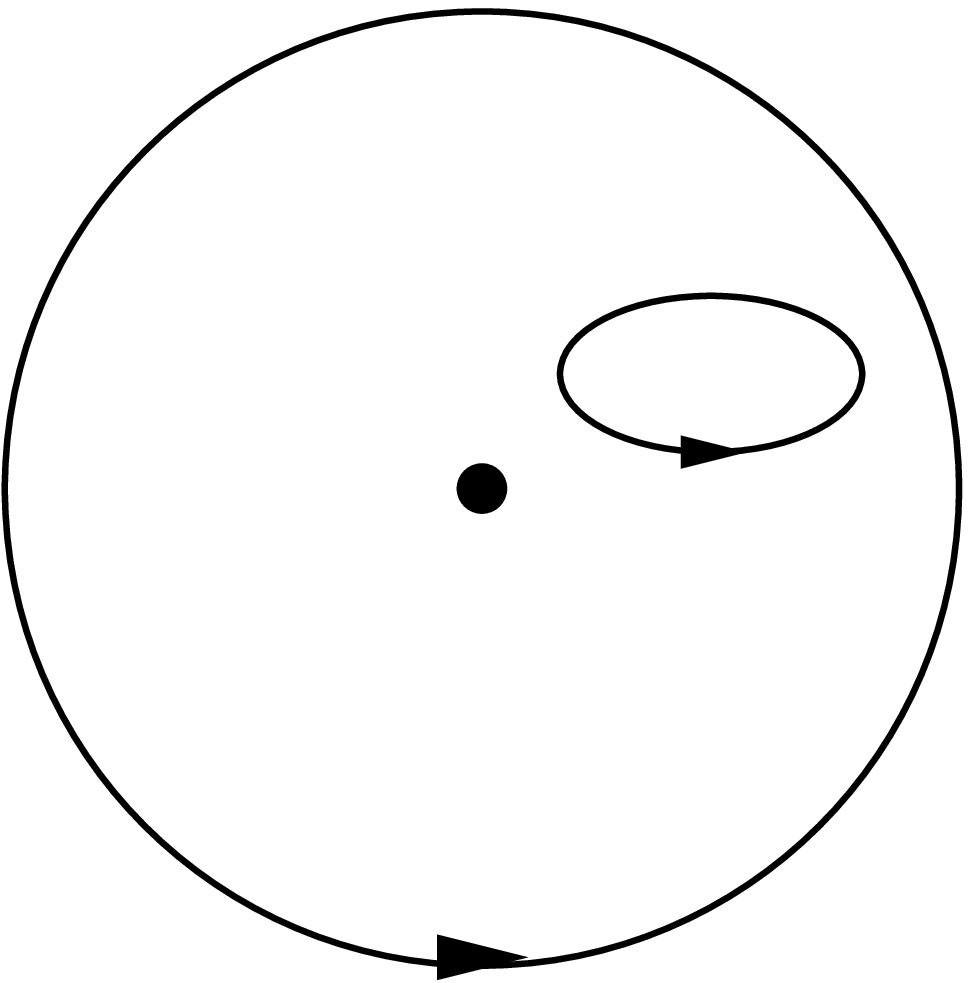} {.200}}
\def\deltaIdentity{\pic{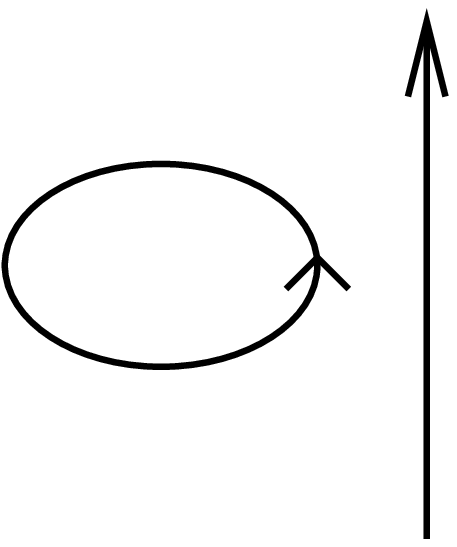} {.300}}

\def\torusknot{\pic{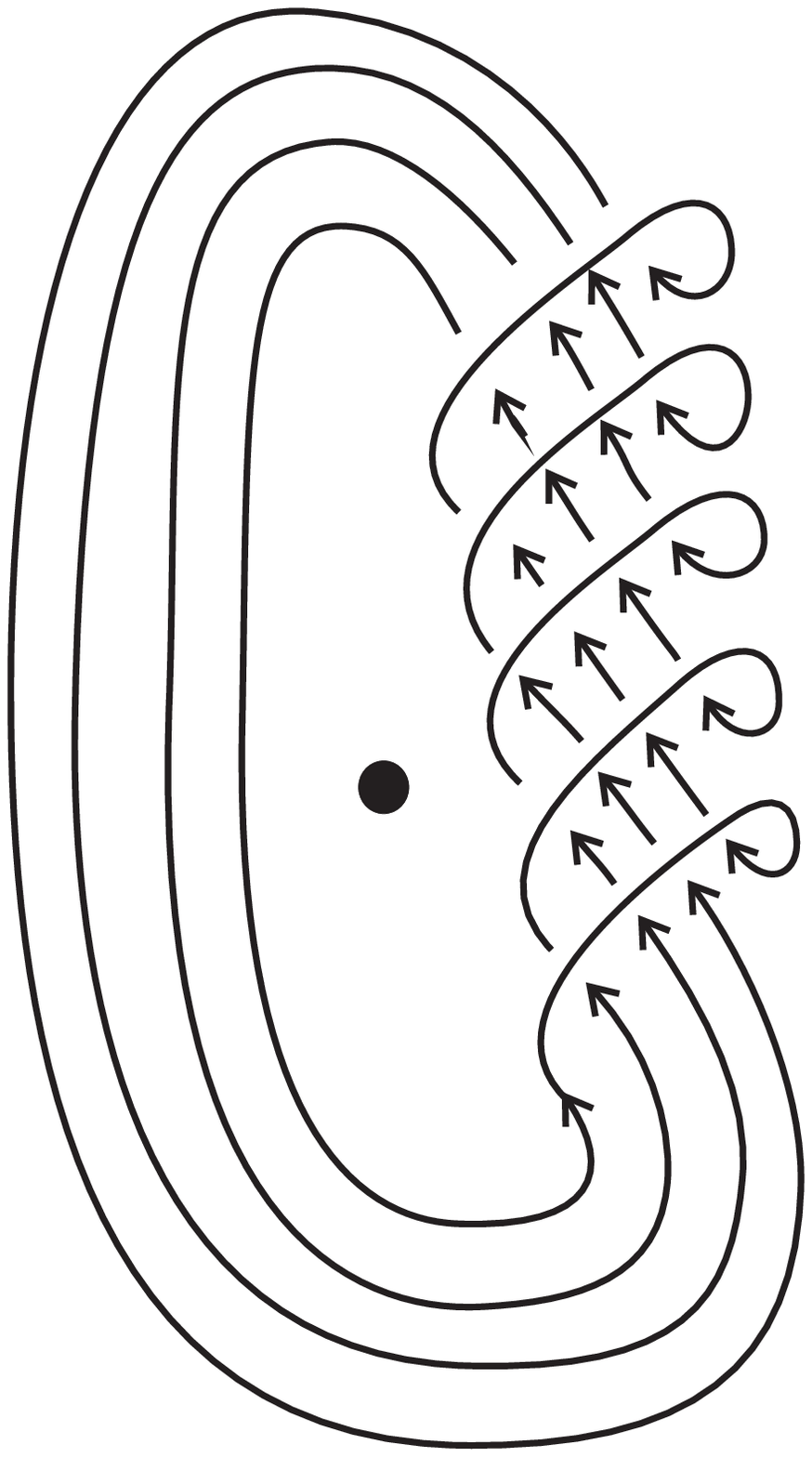} {.200}}
\def\framechange{\pic{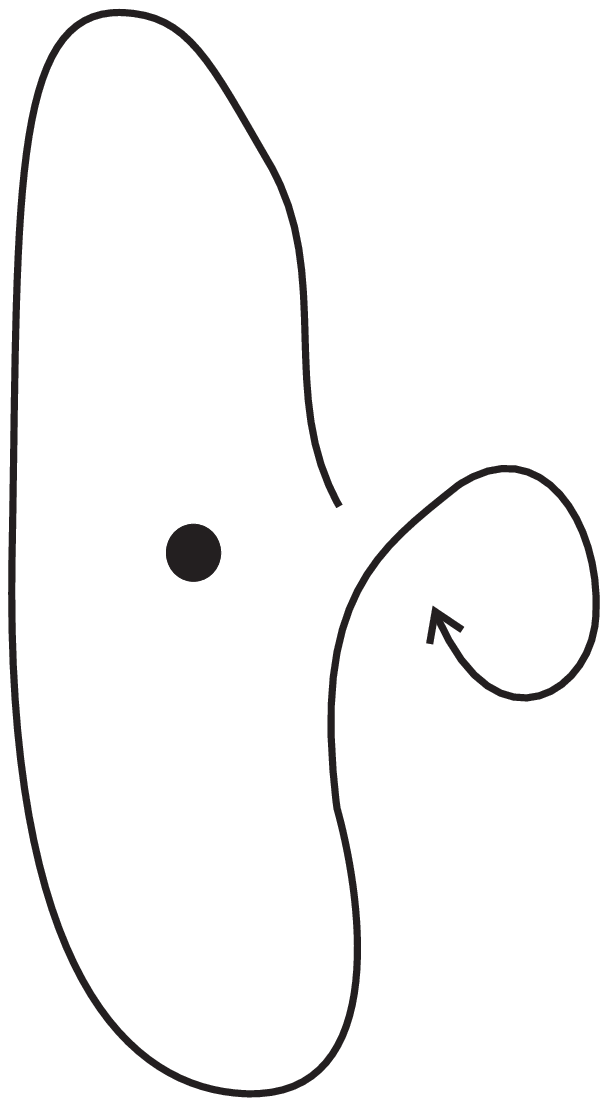} {.200}}

\def\Xor{\pic{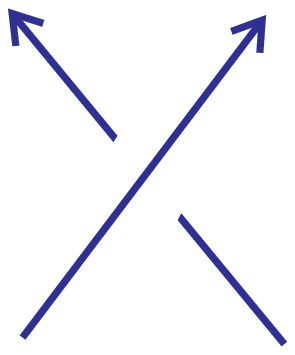} {.250}}
\def\Yor{\pic{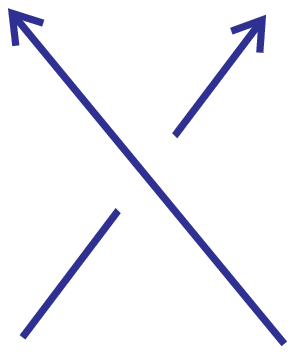} {.250}}
\def\Ior{\pic{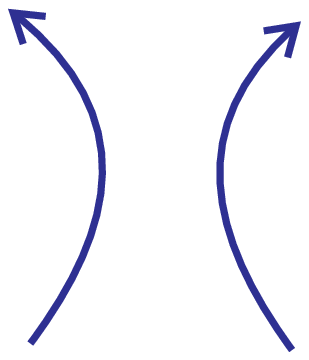} {.250}}

\def\unknot{\pic{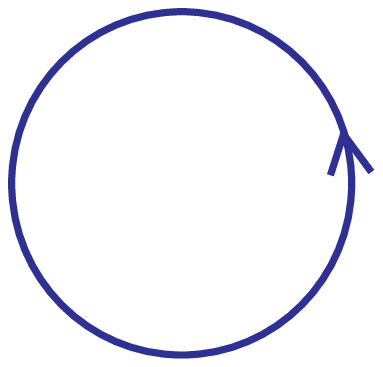} {.150}}

\def\OtherRepresentationPm{\pic{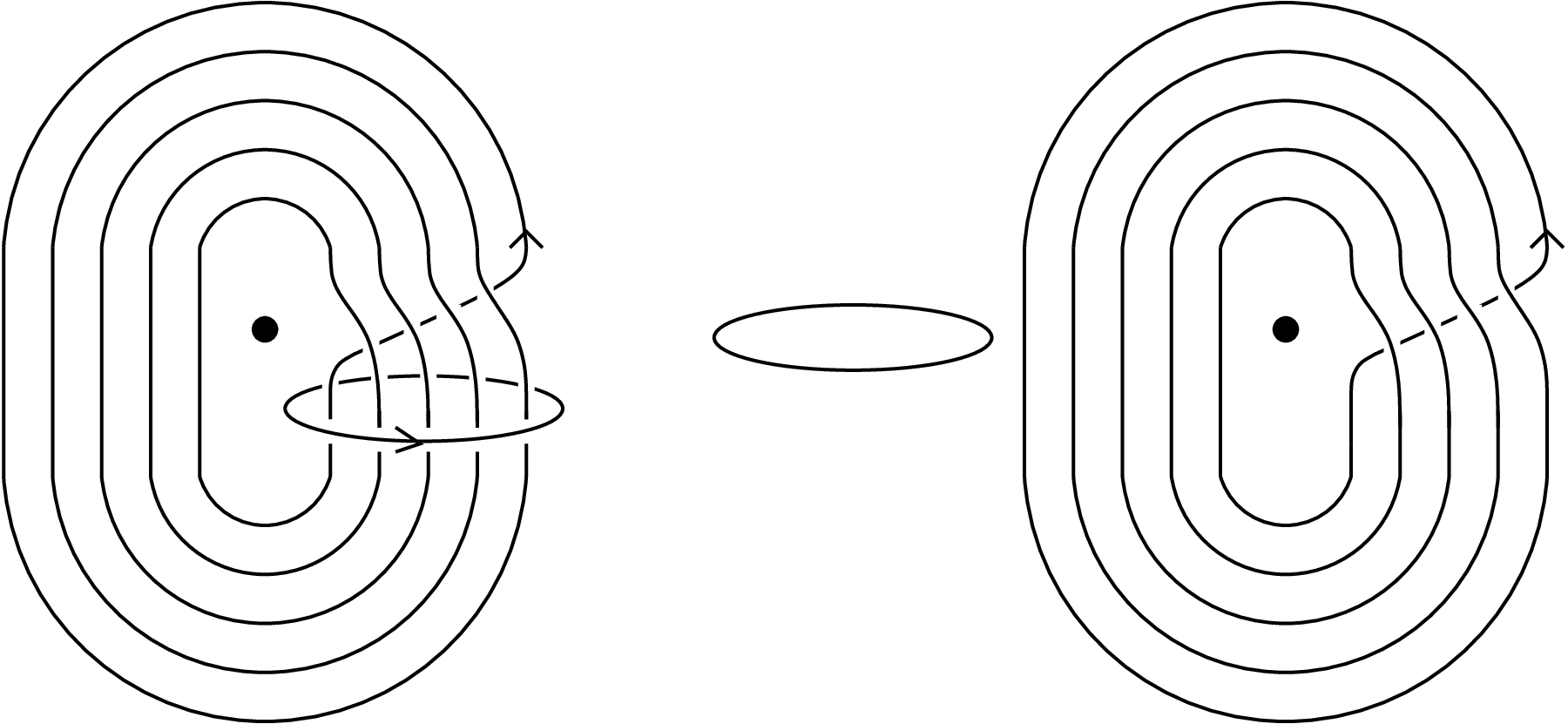}{.250}}

\def\Satellite{\pic{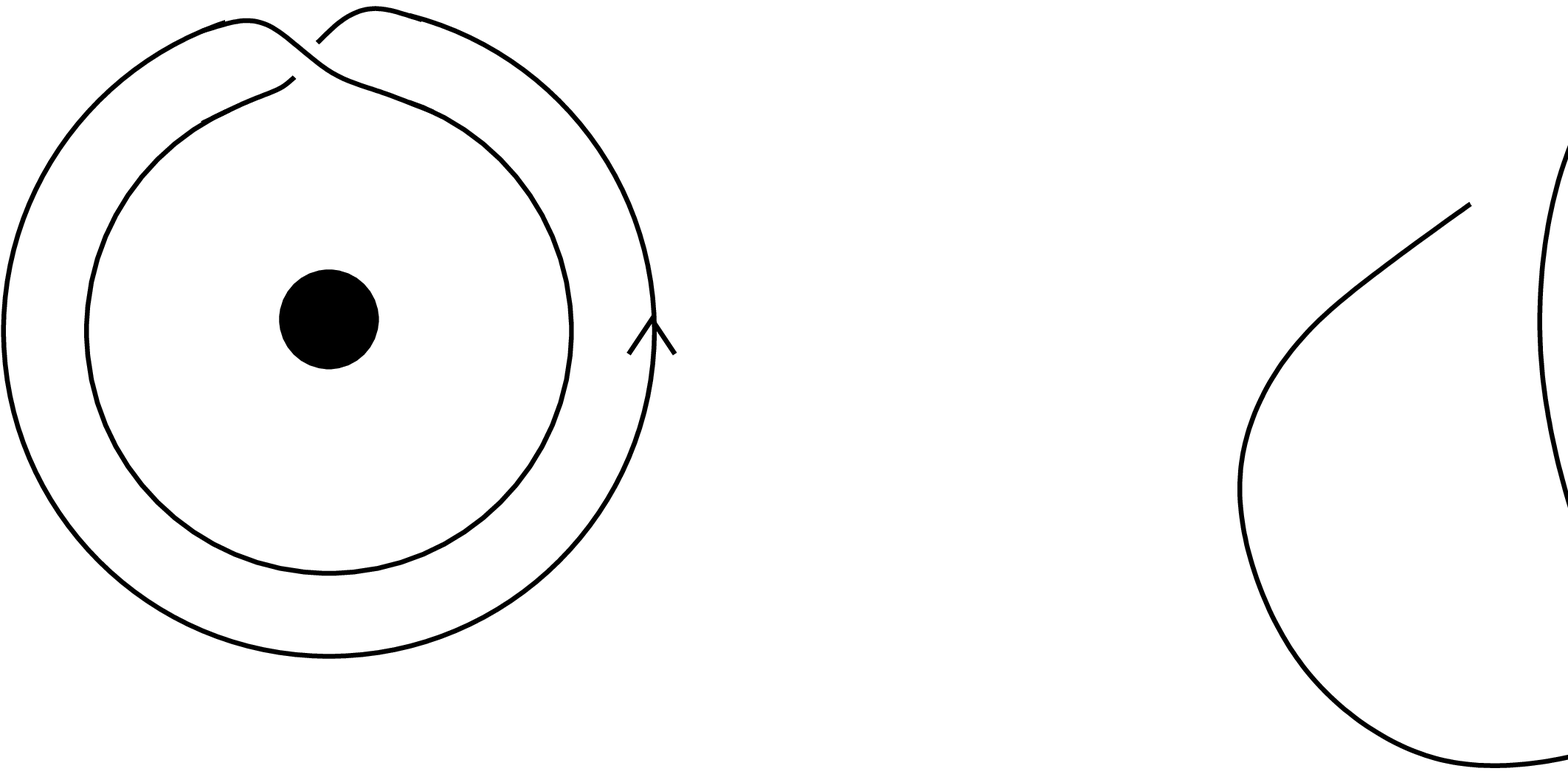}{.200}}

\def\Rmatrix{\pic{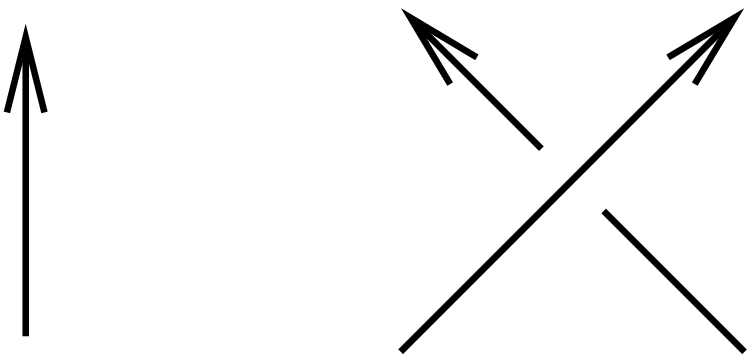} {.250}}

\def\Idor{\pic{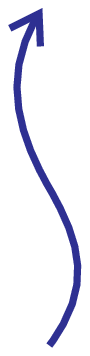} {.250}}
\def\Rcurlor{\pic{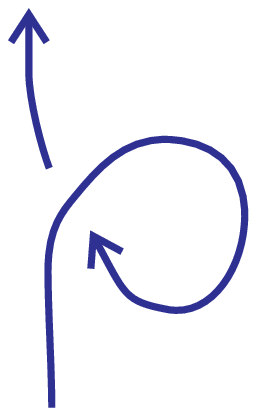} {.250}}
\def\Lcurlor{\pic{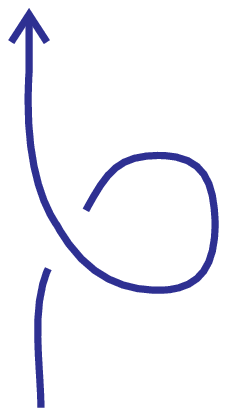} {.250}}

\def\hook{\pic{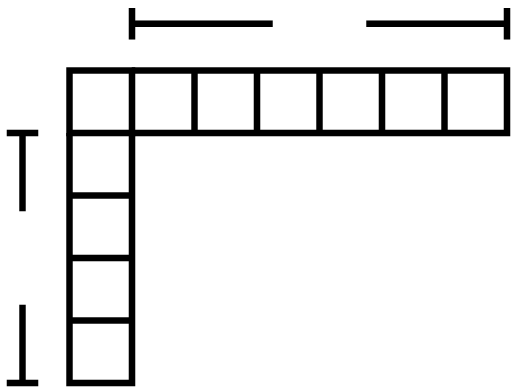} {.40}}
\def\betam{\pic{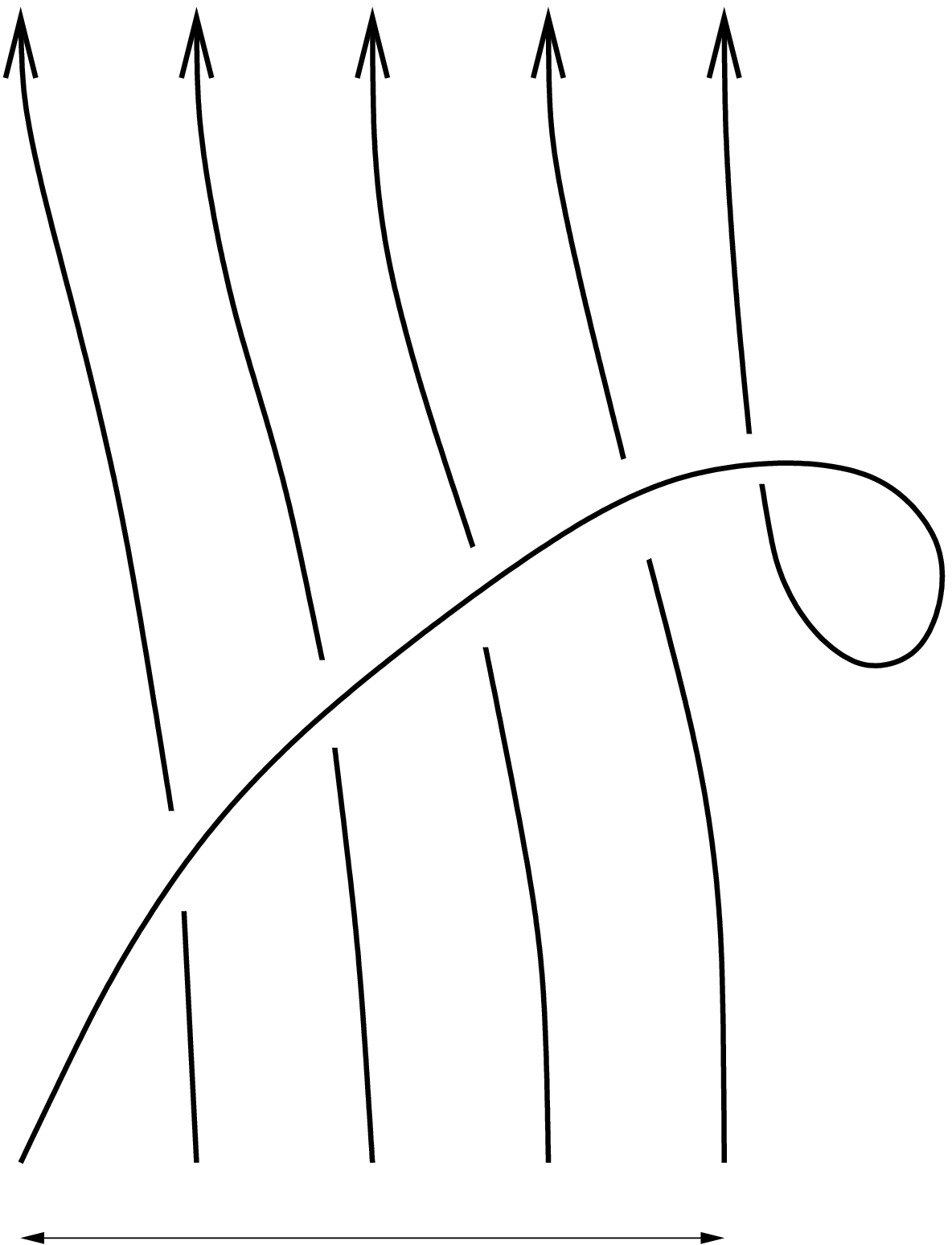}{.15}}

\def\framingmap{\pic{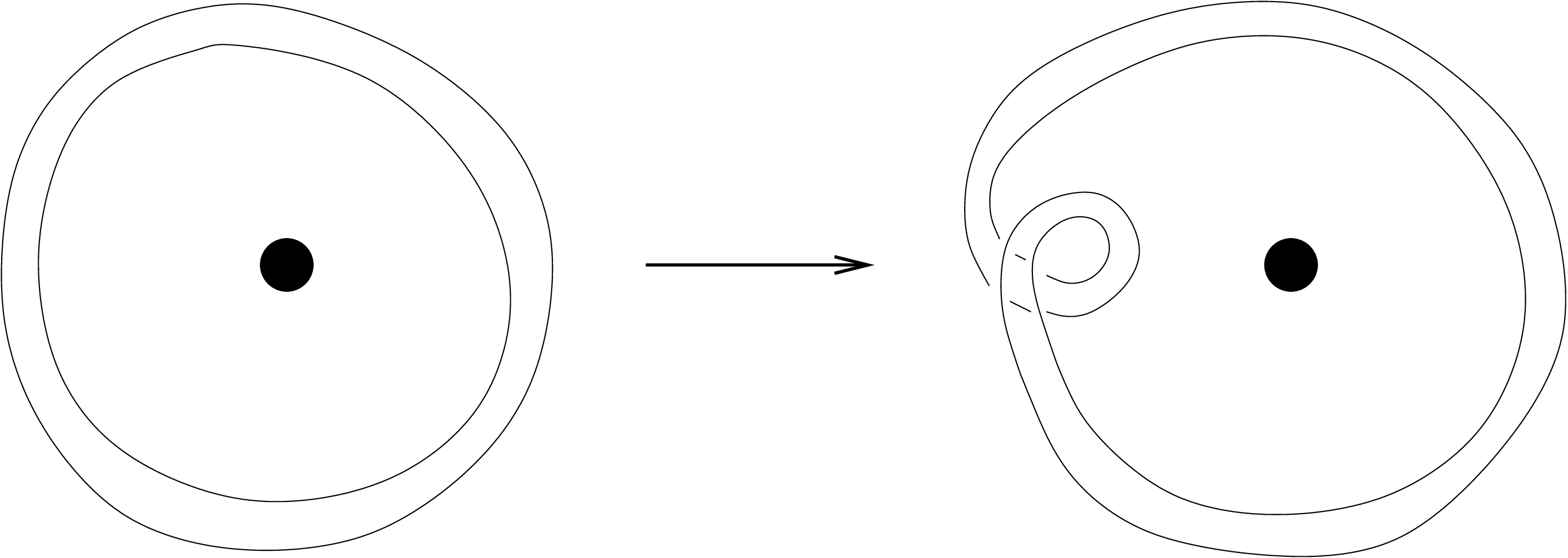}{.15}}
\def\Tiblackboardframing{\pic{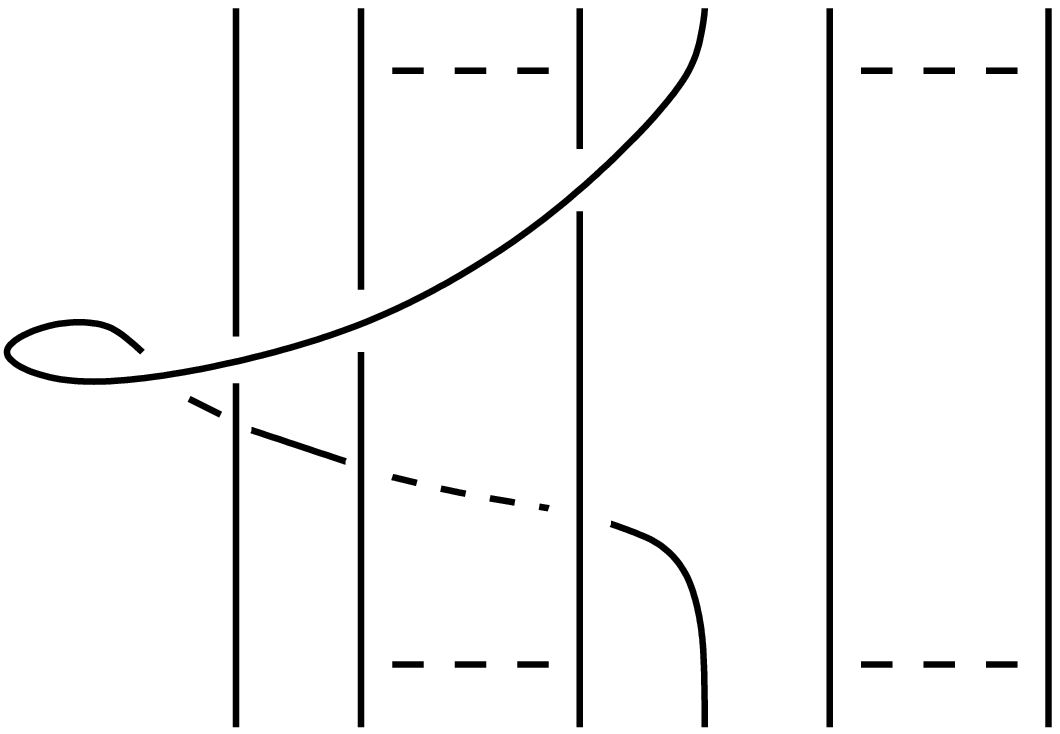} {.300}}

\def\CC{\mathcal {C}}

\newcommand{\bibname}{}

\usepackage{graphicx}
\usepackage{amsmath}
\usepackage{amssymb}
\usepackage{amsthm}

\newcommand{\bc}{\begin{center}}
\newcommand{\ec}{\end{center}}

\newcommand{\be}{\begin{equation}}
\newcommand{\ee}{\end{equation}}

\newcommand{\beqn}{\begin{eqnarray*}}
\newcommand{\eeqn}{\end{eqnarray*}}

\newcommand{\gl}{\lambda}
\newcommand{\qlm}{Q_{\gl,\mu}}
\newcommand{\ds}{\displaystyle}
\newcommand{\x}{\times}

\newtheorem{theorem}{Theorem}
\newtheorem{corollary}[theorem]{Corollary}
\newtheorem{lemma}[theorem]{Lemma}

\newenvironment{remark}{\par\smallskip%
\noindent\textbf{Remark.}\  }%
{\par\smallskip}

\newenvironment{example}{\par\smallskip%
\noindent\textbf{Example.}\  }%
{\par\smallskip}
\renewenvironment{proof}[1][Proof]{\textit{#1.} }{\hfill \rule{0.5em}{0.5em}}

\begin{document}

\bc{\Large\bf Geometrical relations and plethysms\\ in the Homfly skein of the annulus\\[3mm]}
{\sc H. R. Morton {\rm and } P. M. G. Manch\'on\\[2mm]}
 {\small \sl Department of Mathematical Sciences\\ University of Liverpool\\ Peach Street, Liverpool L69 7ZL\\[2mm]
Department of Applied Mathematics, EUIT Industrial
\\ Universidad Polit\'ecnica de Madrid \\
Ronda de Valencia 3, 
28012 Madrid
}
\ec

\begin{abstract}
Let ${\cal C}_m$ be the closure of the Hecke algebra with $m$ strings $H_m$ in the oriented framed Homfly skein ${\cal C}$ of the annulus \cite{{skein},{HadjiMorton},{Sascha},{AistonMorton}}, which provides the natural parameter space for the Homfly satellite invariants of a knot. The submodule ${\cal C}_+ \subset {\cal C}$ spanned by the union $\cup _{m \geq 0} \,{\cal C}_m$ is an algebra, isomorphic to the algebra of the symmetric functions. 
Turaev's geometrical basis for $\CC_+$ consists of monomials in closed $m$-braids $A_m$, the closure of the braid $\sigma _{m-1}\cdots \sigma _2 \sigma _1$. 

We collect and expand formulae relating elements expressed in terms  of symmetric functions to Turaev's basis. We reformulate the formulae of Rosso and Jones for quantum $sl(N)$ invariants of cables \cite{Rosso} in terms of plethysms of symmetric functions, and use the connection  between quantum $sl(N)$ invariants and the skein $\CC_+$  to give a formula for the satellite of a cable as an element of the Homfly skein $\CC_+$. We can then analyse the  case where a cable is decorated by the pattern $P_d$  which corresponds to  a power sum in the symmetric function interpretation of $\CC_+$  to get  direct relations between the Homfly invariants of some diagrams decorated by power sums.
\end{abstract}

\section{Introduction}
The skein $\CC$ of the annulus provides the natural parameter space for organising a large collection of invariants of knots and links, collectively known as their {\em Homfly satellite invariants}. There is a 2-variable invariant $P(K;Q)\in {\bf Z}[v^{\pm1},s^{\pm1}]$ of a framed knot $K$ for each $Q\in\CC$, obtained as the Homfly polynomial of the satellite knot $K*Q$ with companion $K$ and pattern~$Q$.

 The skein $\CC$ has a natural structure as a commutative algebra, leading to several different ways of describing its elements, and consequently the resulting link invariants. For example,
one basis $\{\qlm\}$ for $\CC$ gives a ready translation to the quantum $sl(N)$ invariants of $K$, which are 1-parameter Laurent polynomials, determined by irreducible $sl(N)$ modules.

A subalgebra $\CC_+\subset\CC$ can be interpreted as the ring of symmetric functions in  infinitely many variables $x_1,\ldots,x_N, \ldots$, as described for example in \cite{skein}, and in this context the Schur functions $s_\gl(x_1,\ldots,x_N,\ldots)$ coincide with the basis elements $Q_{\gl,\phi}=Q_\gl$ above.
In the same spirit the full algebra $\CC$ can be interpreted as the ring of supersymmetric functions in variables $\{x_i\}$ and $\{x^*_j\}$.

The skein $\CC$ was originally studied by Turaev \cite{Turaev}, who showed that it is a free polynomial algebra on a doubly infinite sequence of closed braids $\{A_m\}$ and $\{A^*_m\}$.
The subalgebra $\CC_+$  is generated by the braids $\{A_m\}$ alone. In the symmetric function interpretation each $A_m$ is homogeneous of degree $m$.   The monomials $A_\gl=A_{\gl_1}A_{\gl_2}\cdots A_{\gl_k}$, as $\gl\vdash m$ runs through partitions of $m$ with parts $\gl_1\ge\gl_2\ge\cdots\ge\gl_k>0$,  form a geometrically flavoured basis for the linear subspace $\CC_m$ corresponding to the symmetric functions of degree $m$.

In the initial part of this paper we gather together and extend some of the combinatorial formulae relating Turaev's geometric basis $\{A_\gl\}_{\gl\vdash m}$ of $\CC_m$ to the elements representing the complete and elementary symmetric functions $h_m$ and $e_m$ and the power sum $P_m=\sum x_i^m$ of degree $m$.
The combinatorial properties of monomials $\{h_\gl\},\{e_\gl\}$ and $\{P_\gl\}$ in any of these functions are well-documented, for example by Macdonald \cite{Macdonald}, as are the   Jacobi-Trudi and other formulae relating the Schur functions $\{s_\gl\}$ to these.

We derive expressions for Aiston's more geometric representative $X_m$ for the power sum $P_m$ (theorem~\ref{Xm}) and for the mirror image $\overline{A}_m$ of $A_m$ (theorem~\ref{Ambarra}) in Turaev's basis $\{A_\gl\}$, following the skein theoretic arguments in \cite{{skein},{power}} combined with simple manipulation of formal power series.

There is a compact power series formula (\ref{eqnAH}), established  in \cite{skein}, giving $\{A_m\}$ in terms of  the complete symmetric functions $\{h_j\}$.  We use this formula to give a reverse transition expressing $h_m$ in Turaev's basis in theorem~\ref{HfromA}, although  we would have liked to get a tidier form for the coefficients which are given in closed form by lemma~\ref{theta}. In some sense this is one of the more extreme transitions from the geometric to the representation theoretic; the power sums $X_m$ or $P_m$ provide an intermediate state which has a foot in both camps, and correspondingly the transitions between  these elements and either $\{A_m\}$ or $\{h_m\}$ have a much easier form.

The formula (\ref{eqnAH}) can also be used directly to express $A_m$ in terms of the Schur functions, and hence in terms of the skein elements $\{Q_\gl\}$, resulting in a simple deformation of the combinatorial expression for the power sum $P_m$ as an alternating sum of $m$-hooks in theorem \ref{Ahook}. This formula can be derived from the work of Rosso and Jones, \cite{Rosso}, and was used by Aiston \cite{Aiston} in her original construction of a geometric representative for $P_m$.

 In the later part of the paper we interpret the descriptions of  Rosso and Jones, \cite{Rosso}, about quantum invariants of cables in the case of $sl(N)$, in terms of plethysms of symmetric functions. This involves the decoration of one of the   closed braids $T^n_m$ representing the $(m,n)$ torus knot by an element $Q\in\CC$ to form an element $T^n_m*Q$, and allows us to express $T^n_m*Q$ as itself an element of $\CC$,  in theorem \ref{skeinplethysm}. We apply this in the case where $Q$ is a power sum $P_d$ to establish a geometric relation between certain diagrams decorated by power sums in theorem \ref{powerHopf},
originally conjectured in work by the first author with Garoufalidis.

  Interest in Homfly power sum invariants of links, where  all components are decorated by power sums, has been stimulated by the work of Labastida and Mari\~ no \cite{LM}, following the conjectures of Ooguri and Vafa about the integrality of certain combinations of these invariants \cite{OV}.  The fact that power sums can be represented in terms of a small number of closed braids or tangles has given some hope that they may collectively have nice skein theoretic properties, and the results here represent some limited success in their understanding.

\subsection*{The organisation of the paper}

Section \ref{homfly} gives a brief account of Homfly skein theory, including the skein-theoretic model for Hecke algebras of type $A$ and their extensions to the skein of the annulus, $\CC$, related to these by the geometric operation of closure of braids and tangles.  We define the mirror map in a Homfly skein, and then describe Turaev's closed braid basis $\{A_\gl\}$ for the subalgebra $\CC_+$ of the skein $\CC$, following closely the account in \cite{skein}.

In section \ref{geometric} we introduce the geometric elements $X_m$ and derive   formulae for $X_m$ and the mirror image $\overline{A}_m$ in terms of Turaev's basis. The skein theory arguments, following largely the account in \cite{power}, lead to our formal power series derivation of the formulae.

Section \ref{symmetric} summarises the   
interpretation  of $\CC_+$ as symmetric functions.   We discuss the  representation of the complete symmetric functions $h_m$, the elementary symmetric functions $e_m$ and the power sums $P_m$, following \cite{skein}.  The power series equation (\ref{eqnAH}) relating the closed braids $\{A_m \}$ and the complete symmetric functions from \cite{skein} is inverted to provide a formula for $h_m$ in terms of Turaev's basis $\{A_\gl\}$. This gives immediately a corresponding  formula for $e_m$, complemented by a formula for $P_m$ arising from its close relation with $X_m$, given in equation (\ref{eqnPA}).

Section \ref{hook} introduces the meridian maps and the skein theoretic representatives $\{Q_\gl\}$ for the Schur functions $\{s_\gl\}$,
following Lukac \cite{Sascha}. In theorem \ref{Ahook} we expand equation (\ref{eqnAH})  using symmetric functions to express $A_m$ in the basis $\{Q_\gl\}$ of Schur functions.

Section \ref{quantum} describes the relation between the Homfly satellite invariants of a knot and its quantum $sl(N)$ invariants, using the skein theoretic model of the Hecke algebras and their idempotents in \cite{AistonMorton}, and the work of Lukac, \cite{Lukacthesis}.

Section \ref{cable} shows how to interpret the work of Rosso and Jones on quantum invariants of cables in terms of decoration of the cable patterns $T^n_m$ by elements of $\CC_+$, initially using the elements $Q_\gl$. 

Section \ref{powercable} applies this to show how a power sum $P_M$ in the skein $\CC$ enclosed by a meridian decorated  by another power sum $P_N$ can be represented  in $\CC$ by a simple closed braid decorated by the power sum $P_d$ where $d={\gcd(M,N)}$.

We conclude with a number of consequences of this result in section \ref{Ex}.

\section{Homfly skein theory} \label{homfly}

For a surface $F$ with some designated input and output boundary points
the (linear) {Homfly skein} of $F$ is defined as linear combinations of oriented diagrams in $F$, up to Reidemeister moves II and III,  modulo the skein relations

\bc\begin{enumerate}
\item \qquad{$\Xor\  -\ \Yor \qquad =\qquad{(s-s^{-1})}\quad\ \Ior \ ,$}
\vspace*{2mm}
\item \qquad
{$ \Rcurlor \quad=\quad {v^{-1}}\quad \Idor\ ,\qquad \Lcurlor \quad=\quad {v}\quad \Idor. $}
\end{enumerate}
\ec

It is an immediate consequence that 
\bc \Idor \unknot \quad=\quad $\delta$ \ \Idor,\ec
where $\delta =\ds\frac{v^{-1}-v}{s-s^{-1}}\in\Lambda$.
The coefficient ring $\Lambda$ is  taken as $Z[v^{\pm1},s^{\pm1}]$, with denominators $\{r\}=s^r-s^{-r}, r\ge1$.

The skein of the annulus is denoted by ${\cal C}$. It becomes a commutative algebra with a product induced by placing one annulus outside another. 

The skein of the rectangle with $m$ inputs at the top and $m$ outputs at the bottom is denoted by $H_m$. We define a product in $H_m$ by stacking one rectangle above the other, obtaining the Hecke algebra $H_m(z)$, when $z=s-s^{-1}=\{1\}$ and the coefficients are extended to $\Lambda$. The Hecke algebra $H_m$ can be also seen as the group algebra of Artin's braid group $B_m$ generated by the  elementary braids $\sigma _i$, $i=1, \dots , m-1$, modulo the further quadratic relations $\sigma _i^2=z\sigma _i +1$.

The closure map from $H_m$ to ${\cal C}$ is the $\Lambda$-linear map induced by considering the closure $\hat{T}$ of a tangle $T$ in the annulus (see figure \ref{closuremap}). The image of this map is denoted by ${\cal C}_m$. 
\begin{figure}[ht!]
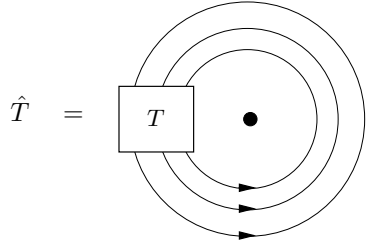

\labellist
\small
\pinlabel $T$ at 58 185
\endlabellist
  \begin{center}
$\hat{T}\quad =\quad \closure $
  \end{center}
\caption{The closure map}\label{closuremap}
\end{figure}

The mirror map in the skein of $F$, defined as in \cite{skein}, is the conjugate linear  involution $(\overline{\phantom{w} })$ on the skein of $F$ induced by switching all crossings on diagrams and inverting $v$ and $s$ in $\Lambda$. We will use it mainly in the skein $\CC$, noting also that $\overline{z}=-z$.

The linear subspace ${\cal C}_m$ has a useful interpretation as the space of symmetric polynomials of degree $m$ in variables $x_1, \dots , x_N$ for large enough $N$.  Moreover, the submodule ${\cal C}_+ \subset {\cal C}$ spanned by the union $\cup _{m \geq 0}\, {\cal C}_m$ is a subalgebra of ${\cal C}$ isomorphic to the algebra of the symmetric functions (see section \ref{symmetric}). 

We now describe Turaev's geometrical basis of the skein ${\cal C}_+$. The element $A_m \in {\cal C}_m$ is the closure of the braid $\sigma _{m-1}\cdots \sigma _2 \sigma _1 \in H_m$, and $\overline{A}_m$ is its mirror image (see figure \ref{Am}). 
\begin{figure}[ht!]
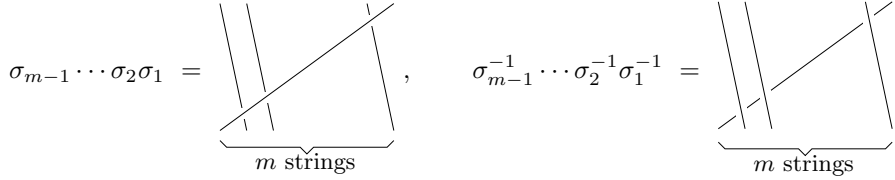
 
\labellist\small
\pinlabel{$m$ strings} at 234 -25
\endlabellist
\bc $\sigma _{m-1}\cdots \sigma _2 \sigma _1 \ =\ $ \Am \ ,
\labellist \small
\pinlabel{$m$ strings} at 234 -25
\endlabellist
 $\qquad\sigma _{m-1}^{-1}\cdots \sigma _2^{-1} \sigma _1^{-1} \ =\ $ \Ambarra \ 
\ec
\vspace*{3mm}
\caption{$A_m$ (resp. $\overline{A}_m$) is the closure of $\sigma _{m-1}\cdots \sigma _2 \sigma _1$ (resp. $\sigma_{m-1}^{-1}\cdots \sigma _2^{-1} \sigma _1^{-1}$)}\label{Am}
\end{figure}

Given a partition $\lambda =(\lambda _1, \dots , \lambda _l)$ of $m$ with length $l$ (we will just write $\lambda \vdash m$), we define the monomial $A_{\lambda}$ by the formula $A_{\lambda }=A_{\lambda _1}\cdots A_{\lambda _l}$. The monomials $\{ A_{\lambda } \}_{\lambda \vdash m}$ constitute a basis of ${\cal C}_m$ (\cite{Turaev}), and the monomials $\{A_\gl\}$ together form  {\em Turaev's  geometric basis} for $\CC_+$. 

\section{Geometric relations in the skein of the annulus}\label{geometric}
 We define intermediate closed braids $A_{i,j}$ between $A_m$ and $\overline{A}_m$, with $i,j \geq 0$, by successively switching one of the crossings as shown in figure \ref{figAij}.

\begin{figure}[ht!]
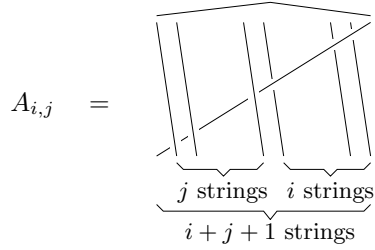

\labellist\small
\pinlabel {$j$ strings} at 210 73
\pinlabel {$i$ strings} at 500 73
\pinlabel {$i+j+1$ strings} at 337 -37
\endlabellist
\bc
$A_{i,j}\quad =\quad $ \Aij
\ec
\vspace*{3mm}
\caption{The closed braid $\ A_{i,j} \ (i,j \geq 0)$} \label{figAij}
\end{figure}

Note that $A_m=A_{m-1,0}$ and $\overline{A}_m=A_{0,m-1}$. 

We define the element $X_m$ in the skein $\CC$, shown in figure  \ref{figXm}, as the sum of $m$ closed $m$-braids,
\[
X_m=\sum_{j=0}^{m-1}A_{m-1-j,j}.
\]

\begin{figure}[ht!]
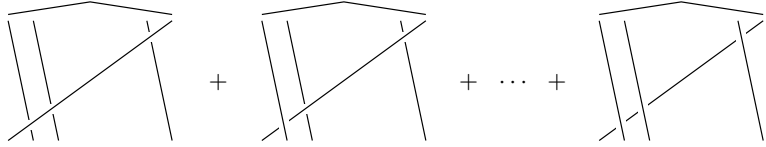

  \begin{center}{
\labellist \small
  \pinlabel{$+$} at 320 90
  \pinlabel{$+ \ \ \cdots \ \ +$} at 766 90
  \endlabellist
 }
    \Xm
    \caption{The element $X_m$} \label{figXm}
  \end{center}
\end{figure}

The elements $\overline{A}_m$ and $X_m$ are readily related to the elements $A_m$ by two formal power series formulae.

Write \[A(t)=1+z\sum_{i=1}^\infty A_i t^i\] and its mirror image \[\overline{A(t)}=1-z\sum_{i=1}^\infty \overline{A}_i t^i.\]

\begin{theorem} \label{AAbar} \[A(t)\overline{A(t)}=1.\]
\end{theorem}

\begin{theorem} \label{AX} \[z\sum_{m=1}^\infty \frac{X_m}{m}t^m=\ln(A(t)).\]
\end{theorem}

These two theorems are consequences of a simple skein-theoretic lemma, originally used by Aiston in \cite{Aiston}.
 We set
$
A_0=\frac{1}{z}$ and $ \overline{A}_0=-\frac{1}{z}$  
to simplify the notation in the following statements.
\begin{lemma} \label{Aij} For $i \geq 1$ we have   $A_{i,j}=zA_i\overline{A}_{j+1}+A_{i-1,j+1}$.
\end{lemma}
\begin{proof} Use the quadratic relation in the skein at the marked crossing to get
\bc
\labellist\small
\pinlabel {$j$ strings} at 210 -25
\pinlabel {$i$ strings} at 500 -25
\endlabellist
$A_{i,j}\quad = \quad \AijMarcado\quad = \quad zA_i\overline{A}_{j+1}+A_{i-1,j+1}.$
\ec
\end{proof}

\begin{lemma} \label{previousAmbarra} For $m\geq 1$ we have 
$
\overline{A}_m=-z\sum_{k=1}^mA_k\overline{A}_{m-k}
$.
\end{lemma}
\begin{proof}  Repeated use of lemma \ref{Aij} gives
\begin{eqnarray*}
  A_m & = & A_{m-1,0} \\
  \   & = & zA_{m-1}\overline{A}_1+A_{m-2,1} \\
  \   & = & zA_{m-1}\overline{A}_1+zA_{m-2}\overline{A}_2+A_{m-3,2} \\
  \   & \vdots & \ \\
  \   & = & zA_{m-1}\overline{A}_1+zA_{m-2}\overline{A}_2+ \dots + zA_1\overline{A}_{m-1}+A_{0,m-1}.
  \   \end{eqnarray*} \label{eqnAij}

The last equation can be written
\[
A_m=z\sum_{k=1}^{m-1}A_{m-k}\overline{A}_k+\overline{A}_m
=z\sum_{k=1}^{m}\overline{A}_kA_{m-k}.
\]

Now apply the mirror map.
\end{proof}
 
\begin{proof}[of theorem \ref{AAbar}] 
The coefficient of $t^m$ in $A(t)\overline{A(t)}$ is
\begin{eqnarray*} &&
-z\overline{A}_m-z^2A_1\overline{A}_{m-1}-z^2A_2\overline{A}_{m-2}-\ldots -z^2A_{m-1}\overline{A}_1+zA_m\\
&=&z\left(  A_m-\overline{A}_m-z\sum_{k=1}^{m-1}A_k\overline{A}_{m-k}\right).
\end{eqnarray*}
This is $0$
for $m\ge1$ by lemma~\ref{previousAmbarra}, while the constant term is $1$.
\end{proof}

\begin{lemma} \label{previousXm}For $m\ge1$ we have  $X_m=-z\displaystyle \sum_{j=1}^{m}jA_j\overline{A}_{m-j}.
$
\end{lemma}
\begin{proof} Sum all the equations   except the first in the proof of lemma \ref{previousAmbarra} to get \[
(m-1)A_m=z((m-1)A_{m-1}\overline{A}_1+(m-2)A_{m-2}\overline{A}_2+...+A_1\overline{A}_{m-1})+\sum_{j=1}^{m-1}A_{m-1-j,j}.
\]
Since $A_m=A_{m-1,0}$ it follows that \[mA_m=z\sum_{j=1}^{m-1}jA_j\overline{A}_{m-j}+X_m.\]
\end{proof}

\begin{proof}[of theorem \ref{AX}]
Since the constant terms in the two series are equal, it is enough to show that their derivatives are equal.

Now \[\frac{d}{dt}\left(z\sum \frac{X_m}{m}t^m\right)=z\sum X_m t^{m-1}\] and 
\[\frac{d}{dt}\left(\ln(A(t))\right)=\frac{A'(t)}{A(t)}=A'(t)\overline{A(t)}\] by theorem \ref{AAbar}.
The coefficient of $t^{m-1}$ in  $A'(t)\overline{A(t)}$ is \[-z^2\displaystyle\sum_{j=1}^m jA_j\overline{A}_{m-j}=zX_m\] by lemma \ref{previousXm}, for all $m\ge 1$, and so the two series are equal.
\end{proof}

We shall use the power series relations to give expressions for $\overline{A}_m$ and $X_m$ in terms of the Turaev basis $\{A_\gl\}_{\gl\vdash m}$ for $\CC_m$. The first of these depends on the general expression for the coefficients $\{d_m\}$ of the inverse, $1+\sum d_m t^m$,  of a formal power series $1+\sum c_n t^n$,  in terms of monomials in the coefficients $\{c_n\}$, while the second, which can be deduced quickly from the first, gives the coefficients of the logarithm of a formal power series. Both of these results can be found by applying the technique given in \cite{Macdonald} (example 11, page 30) for finding the coefficients of the resulting power series when one power series is substituted in another.   

When discussing monomials in the coefficients $\{c_n\}$ it is helpful to distinguish between {\em ordered} monomials, $c_{r_1}c_{r_2}\ldots c_{r_l}$, and the corresponding standard monomial $c_{\gl_1}c_{\gl_2}\ldots c_{\gl_l}$, where the sequence $(r_1, r_2, \dots, r_l)$ is rearranged into descending order $(\gl_1, \gl_2, \dots, \gl_l)$.  The standard monomial can then be described as $c_\gl$ where $\gl$ is the partition of $m=\sum \gl_i$ having $l$ parts $\gl_1\ge\gl_2\ge \dots \ge \gl_l>0$.  We write $k_\gl$ for the number of ordered monomials with standard monomial $c_\gl$, or equally the number of rearrangements of the sequence $(\gl_1, \gl_2, \dots, \gl_l)$.  

The coefficient $d_m$ of $t^m$ in the inverse series is the sum over partitions $\gl\vdash m$  of $(-1)^l k_\gl c_\gl$.

\begin{theorem} \label{Ambarra} For $m\geq 1$ we have that
\[
\overline{A}_m=\sum_{\lambda \vdash m}k_{\lambda }(-z)^{l-1}A_{\lambda}.
\]
\end{theorem}
 \begin{proof} This follows at once from theorem \ref{AAbar} and the formula for the inverse series.
\end{proof}

There is a simple combinatorial formula for $k_\gl$ as a multinomial coefficient, in terms of the multiplicities of the parts of $\gl$. If $\gl$ has  
$r${ \em distinct} parts repeated $m_1,\ldots, m_r$ times respectively, making a total of $l=m_1+\cdots+m_r$ parts altogether, there are ${l\choose{m_1,\ldots, m_r}}=\ds\frac{l!}{m_1!\cdots m_r!}=k_\gl$ possible rearrangements.

\begin{example}
The partition $\lambda =(4,4,4,2,2,1,1,1)$ with length $l=8$ has three distinct parts with multiplicities $m_1=3, m _2=2$ and $ m_3=3$, hence $k_{\lambda }=\frac{8!}{3! \ 2! \ 3!}=560$. It follows from theorem \ref{Ambarra} that the coefficient of $A_4^3A_2^2A_1^3$ in  $\overline{A}_{19}$ is $-560z^7$.
\end{example}

\begin{theorem} \label{Xm} $X_m$ is given in terms of the monomials $\{ A_{\lambda} \} _{\lambda \vdash m}$ by the formula
\[
X_m=m\sum_{\lambda \vdash m}\frac{k_{\lambda }}{l}(-z)^{l-1}A_{\lambda }.
\]
\end{theorem}
\begin{proof} Differentiate $\ln (A(t))$ with respect to $z$, treating each $A_i$ as constant.

By theorem \ref{AAbar} 
\begin{eqnarray*}\frac{d}{dz}\left(\ln(A(t))\right)&=&\frac{d}{dz}(A(t))/A(t)=\frac{d}{dz}(A(t))\x\overline{A(t)}\\&=&\frac{A(t)-1}{z}\x\overline{A(t)}=\frac{1-\overline{A(t)}}{z}\\
&=&\sum_{m=1}^\infty\overline{A}_m t^m.
\end{eqnarray*}

By theorems \ref{AX} and \ref{Ambarra},  we have \[\frac{d}{dz}(zX_m/m)=\overline{A}_m=\sum_{\gl\vdash m} k_\gl (-z)^{l-1}A_\gl.\]

Integrating the right hand side gives \[\frac{zX_m}{m}= \sum_{\gl\vdash m}k_\gl (-1)^{l-1}\frac{z^l}{l}A_\gl,\] and the theorem follows.
\end{proof}
\begin{example}
For the partition $\lambda =(3,3,1,1,1)$ we have $m=9, l=5, m _1=2$ and $m_2=3$. Then $k_{\lambda }=\frac{5!}{2! \ 3!}=10$ and the coefficient of $A_3^2A_1^3$ in $X_9$ is
\[
m\frac{k_{\lambda }}{l}(-z)^{l-1}= 9\frac{10}{5}(-z)^{5-1}=18z^4.
\]
\end{example}

\section{Symmetric functions} \label{symmetric}

The element $h_m \in {\cal C}_m$, which is taken to represent  the complete symmetric function of degree $m$, is the closure of the element $\frac{1}{\alpha _m}a_m \in H_m$ where $a_m=\sum _{\pi \in S_m}s^{l(\pi )}\omega _{\pi}$ is one of the two basic quasi-idempotent elements of $H_m$. Here $\omega _{\pi}$ is the positive permutation braid associated to the permutation $\pi \in S_m$ with length $l(\pi )$ and $\alpha_m$ is given by the equation $a_ma_m=\alpha _ma_m$ \cite{{Sascha},{AistonMorton},{skein}}.  Using the other quasi-idempotent $b_m=\sum _{\pi \in S_m}(-s)^{-l(\pi )}\omega _{\pi}$ in a similar way determines the element $e_m$ which represents the elementary symmetric function.  These elements are related by the power series equation $H(t)E(-t)=1$, where $H(t)=1+\sum h_n t^n$ and $E(t)=1+\sum e_n t^n$. The involution on the skein $\CC$ induced by sending each diagram to itself and altering the coefficients by fixing $v$ and interchanging $s$ with $-s^{-1}$ will interchange $h_m$ and $e_m$.  

The subalgebra $\CC_+$ of $\CC$ is generated as an algebra by $\{h_n\}$, and the monomials $h_\gl$ of weight $m$, where $\gl\vdash m$, form a basis for $\CC_m$, allowing $\CC_+$ to be interpreted as the ring of symmetric functions in variables $x_1,\ldots,x_N,\ldots$ with coefficients in $\Lambda$. In this interpretation $\CC_m$ consists of the homogeneous functions of degree $m$.

The power sums $P_m =\sum x_i^m$ are symmetric functions which can be written in terms of the complete symmetric functions by Newton's power sum relation $\ln H(t)=\sum_{m}\frac{P_m}{m}t^m$. This equation defines $P_m$ as an element of the skein $\CC_m$.  The element $P_m$ is used in \cite{skein} for describing the $m$th power sum of the Murphy operators in $H_n$, independently of $n$. It is shown in \cite{power} that the more geometric element $X_m$ in figure \ref{figXm} is a scalar multiple of $P_m$, given explicitly as  $X_m=[m]P_m$, where $[m]$ is the quantum integer $\ds\frac{s^m-s^{-m}}{s-s^{-1}}$.

Theorem \ref{Xm} gives the immediate expression
\be
P_m=\frac{m}{[m]}\sum_{\lambda \vdash m}\frac{k_{\lambda }}{l}(-z)^{l-1}A_{\lambda } \label{eqnPA}
\ee 
for $P_m$ in Turaev's basis.

The complete symmetric functions $\{h_n\}$ themselves are shown in theorem~3.6 of \cite{skein} to be related to Turaev's closed braids $\{A_m\}$
by the equation
\be
A(t)=\frac{H(st)}{H(s^{-1}t)}. \label{eqnAH}
\ee

We now derive an expression for $h_m$ in terms of Turaev's basis. We had hoped for a more illuminating way to display the coefficient of $A_\gl$ in terms of the partition $\gl\vdash m$, but we do have an explicit rational function in $\Lambda$ whose numerator may be able to be reorganised better in some given cases.

\subsection{The complete symmetric functions $h_m$} 
Equation (\ref{eqnAH}) can be written in the form $H(s^{-1}t)A(t)-H(st)=0$. Equivalently
\[
\left( \sum_{n=0}^{\infty }s^{-n}h_nt^n\right) \left( 1+z\sum _{m=1}^{\infty }A_mt^m\right) -\sum_{n=0}^{\infty }s^nh_nt^n=0.
\]
Considering the coefficient of $t^m$ we obtain the equation
\[
s^{-m}h_m+z\sum _{j=0}^{m-1}s^{-j}h_jA_{m-j}-s^mh_m=0,
\]
hence
\[
[m]h_m=\sum _{j=0}^{m-1}s^{-j}h_jA_{m-j}.
\]

For a partition $\lambda =(\lambda _1 , \dots , \lambda _l)$ of $m$, we will write $l_i$ for the multiplicity of the part $\lambda _i$ in $\lambda$. We will also write $\lambda -\lambda _i$ for the partition $(\dots ,\lambda _{i-1}, \lambda _{i+1}, \dots )\vdash m-\lambda _i$ with length $l-1$. 
\begin{theorem} \label{HfromA} The complete symmetric functions can be written as
\[
h_m=\sum_{\lambda \vdash m }\theta _{\lambda } \ A_{\lambda },
\]
where $\theta _{\lambda }$ is given recursively by the formula
\[
\theta _{\lambda }=\frac{1}{s^m[m]}\ \sum_{i=1}^l\frac{s^{\lambda _i}}{l_i}\ \theta _{\lambda -\lambda _i}
\]
and $\theta _{\emptyset}=1$, where $\emptyset$ denotes the empty partition.
\end{theorem}
\begin{proof} We prove the theorem by induction on $m$. For $m=1$ we have that $\lambda =(1)$ is the only partition of $1$ and $\theta _{(1)}=1$, hence the formula just says that $h_1=A_1$.  
Assume the theorem for $1, \dots , m-1$. Then
\begin{eqnarray*}
[m]h_m &=& \sum _{j=0}^{m-1}s^{-j}h_jA_{m-j} \\
\ &=& \sum _{j=0}^{m-1}s^{-j}A_{m-j}\left( \sum _{\mu \vdash j}\theta _{\mu } \ A_{\mu _1}\cdots A_{\mu _l}\right) \ \ \ \ \mbox{(by induction)} \\
\ &=& \sum _{\lambda \vdash m}\left( \sum _{i=1}^l \frac{s^{-(m-\lambda _i)}}{l_i} \ \theta _{\lambda -\lambda _i}\right) A_{\lambda _1}\cdots A_{\lambda _l},
\end{eqnarray*} 
and $\theta _{\lambda }$ is obviously the expression in brackets divided by $[m]$.
\end{proof} 
\begin{remark} The coefficient of $A_1^m$ in $h_m$ is $\theta _{(1,\stackrel{m}{\dots}, 1)}=\frac{1}{\alpha _m}$, where $\alpha _m=s^{\frac{m(m-1)}{2}}[m]!$ satisfies the equation $a_m^2=\alpha _ma_m$.  Also the coefficient of $A_m$ in $h_m$ is $\theta_{(m)}=\frac{1}{[m]}$. \end{remark} 
 
We now provide a non-recursive formula for the coefficients $\{\theta_\gl\}$ of theorem~\ref{HfromA}. First, we   introduce some notation: if $\mu =(\mu _1, \dots , \mu _l)$ is a (not necessarily decreasing) finite sequence of integers $\mu _j >0$, we define the coefficient
\[
c_{\mu }= 
\prod _{i=1}^l\frac{1}{[\mu _1+\dots +\mu _i]s^{\mu _1+ \dots +\mu _i}}.
\]
If $\lambda =(\lambda _1, \dots , \lambda _l)$ is a partition with length $l$ and $\alpha \in S_l$, the set of permutations of $\{ 1,2,\dots ,l\}$, we will write $\lambda _{\alpha }$ for the finite sequence $\lambda _{\alpha }=(\lambda _{\alpha(1)}, \dots , \lambda _{\alpha (l)})$.

\begin{lemma} \label{theta} For a partition $\lambda $ of $m$ with length $l$ the coefficient $\theta_\gl$ can be written as
\[
\theta _{\lambda } = \frac{k_{\lambda }}{l!}s^m\sum _{\alpha \in S_l}c_{\lambda _{\alpha}}.
\]
\end{lemma} 
\begin{proof} By induction on the length $l$ of the partition $\lambda$. If $l=1$ we have $\theta _{\lambda }=\theta _{(m)}=\frac{1}{[m]}$, and the right hand side is
\[
\frac{k_{(m)}}{1!}s^m\frac{1}{[m]s^m}=\frac{1}{[m]}.
\]
Assume now the formula for $1, \dots ,l-1$ and consider a partition $\lambda =(\lambda _1, \dots , \lambda _l) \vdash m$ with length $l$. By definition
\[
\theta _{\lambda}=\frac{1}{s^m[m]}\sum_{i=1}^l\frac{s^{\lambda _i}}{l_i}\theta _{\lambda -\lambda _i}
\]
and by induction (the partitions $\lambda -\lambda _i$ have length $l-1$), we have that
\[
\theta _{\lambda }=
\frac{1}{s^m[m]}\sum_{i=1}^l \left( \frac{s^{\lambda _i}}{l_i}\frac{k_{\lambda -\lambda _i}}{(l-1)!}s^{m-\lambda _i} \sum_{\beta \in S_{l-1}} c_{(\lambda -\lambda _i)_{\beta}}\right) .
\]
Since $k_{\lambda -\lambda _i}=k_{\lambda }\frac{l_i}{l}$, we deduce that
\[
\theta _{\lambda }=\frac{k_{\lambda }}{l! [m]}\sum_{i=1}^l \sum_{\beta \in S_{l-1}} c_{(\lambda -\lambda _i)_{\beta}}.
\]
For every $1 \leq i \leq l$ and permutation $\beta \in S_{l-1}$ we define the permutation $\alpha \in S_l$ as the composite permutation $\alpha=\beta \ (l, i, i+1, \dots , l-1) $ which maps $l$ to $i$, establishing a bijection between $\{ 1, \dots , l\} \times S_{l-1}$ and $S_l$. It turns out that $\lambda _i$ is the last part of $\lambda _{\alpha }$, and $\lambda _{\alpha }-(\lambda _\alpha)_l =(\lambda -\lambda _i)_{\beta }$. It follows that $[m]s^mc_{\lambda _{\alpha }}=c_{(\lambda -\lambda _i)_\beta}$, hence
\[
\theta _{\lambda }=\frac{k_{\lambda }}{l! [m]}\sum_{\alpha \in S_l} [m]s^mc_{\lambda _{\alpha}}=\frac{k_{\lambda }}{l!}s^m\sum_{\alpha \in S_l} c_{\lambda _{\alpha}}.
\]
\end{proof}

\begin{example}
We have
\[
\begin{array}{l}
h_1=A_1,\\
h_2=\frac{s}{s^2+1}A_2+\frac{1}{s^2+1}A_1^2,\\
h_3=\frac{s^2}{s^4+s^2+1}A_3
+\frac{s(s^2+2)}{(s^4+s^2+1)(s^2+1)}A_2A_1
+\frac{1}{(s^4+s^2+1)(s^2+1)}A_1^3, \ {\rm etc.} 
\end{array}
\]
 
For example, for $\lambda =(2,1)$, we have $k_{\lambda}=2$, $l=2$, $m=3$, $c_{(2,1)}=\frac{1}{[2][3]s^5}$ and $c_{(1,2)}=\frac{1}{[3]s^4}$, giving the coefficient of $A_2A_1$ in $h_3$ as $\theta _{(2,1)}=\frac{s(s^2+2)}{(s^4+s^2+1)(s^2+1)}$.

In general each coefficient in $h_m$ is a rational function with denominator $[m]!$.  As a further example,  the coefficient $\theta _{(3,3,2)}$ of $A_3^2A_2$ in $h_8$ is \[\ \frac{s^8}{2}(2c_{(3,3,2)}+2c_{(3,2,3)}+2c_{(2,3,3)})=s^8 \left( \frac{1}{s^{17}[3][6][8]}+\frac{1}{s^{16}[3][5][8]}+\frac{1}{s^{15}[2][5][8]} \right) .\]
\end{example}
 
\begin{corollary}\label{EfromA} We have a similar formula for the elementary symmetric functions,
\[
e_m=\sum_{\lambda \vdash m }\tau _{\lambda } \ A_{\lambda },
\]
where, for each partition $\lambda \vdash m$ with length $l$, the coefficient $\tau _{\lambda }$ is
\[
\tau _{\lambda } = (-1)^{m+l}\frac{k_{\lambda }}{l!}s^{-m}\sum _{\alpha \in S_l}
\prod _{i=1}^l\frac{s^{\lambda _{\alpha (1)}+ \dots +\lambda _{\alpha (i)}}}{[\lambda _{\alpha (1)}+\dots +\lambda _{\alpha (i)}]}.
\]
\end{corollary}
\begin{proof}
The element $e_m$ can be obtained from $h_m$ with the substitution $s=-s^{-1}$. After this substitution, $[k]s^k$ becomes $-\frac{[k]}{s^k}$. 
\end{proof}

\section{Schur functions and hook partitions}\label{hook}
The meridian maps, introduced explicitly in \cite{skein}, are linear maps $\varphi,\overline\varphi:\CC\to\CC$, induced by including an oriented meridian around any diagram $X$ in the annulus as shown in figure \ref{figmeridianmap}.

\begin{figure}[ht!]
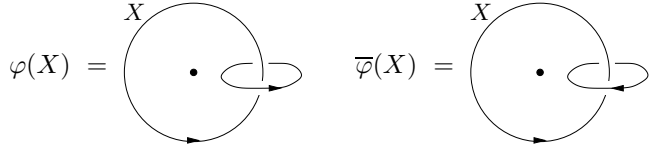

\bc{\labellist
\small
\pinlabel {$X$} at 22 262
\endlabellist}
$\varphi(X)\ =\ \meridianmap\quad\quad\overline{\varphi}(X)\ =\ {\labellist
\small
\pinlabel {$X$} at 22 262
\endlabellist}\meridianmapRev$
\ec 
\caption{The meridian maps} \label{figmeridianmap}
\end{figure}

It is shown in \cite{HadjiMorton} that the eigenvectors of $\varphi$ have no repeated eigenvalues, and that there is a basis $\qlm$ of $\CC$ consisting of these eigenvectors, where $\gl $ and $\mu$ run through all pairs of partitions. The subspace $\CC_m$ has a basis $Q_{\gl,\phi}=Q_\gl$, where $\gl$ runs through partitions of $m$. 

This basis has been identified  by Lukac with the basis formed by the 
 closures $Q_{\lambda }$ of Aiston's  idempotent elements $e_\gl$ in the Hecke algebra $H_m$. Lukac has shown also   that they   represent the Schur functions $s_{\lambda}$ in the interpretation as symmetric functions. Thus they can be expressed  as determinants of the elements $h_m$ by the Jacobi-Trudi formula; precisely, $Q_{\lambda }={\rm det}(h_{\lambda _i+j-i})_{1\leq i,j \leq l}$ if $\lambda =(\lambda _1, \dots , \lambda _l)$ (\cite{{Sascha},{HadjiMorton}}). Since $\overline{h}_i=h_i$ (\cite{skein}, lemma 3.7), the 
 elements $Q_{\lambda}$ are not affected by the mirror map.

As extreme cases we have $Q_{\lambda}=h_m$ when $\lambda =(m)$ is a row partition, and $Q_\gl=e_m$ when $\gl= (1, \stackrel{m}{\dots}, 1)$ is a column partition.
In Frobenius notation, $(a|b)$ denotes the \emph{hook} partition of $m=a+b+1$ with an arm of length $a$ and a leg of length $b$, as shown in figure \ref{fighook}.
 
\begin{figure}[ht!]
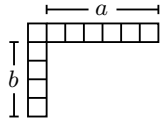

\labellist
\small
\pinlabel {$a$} at 91 105
\pinlabel {$b$} at 6 38
\endlabellist
\bc
\hook\ec
\caption{The hook $(a|b)$} \label{fighook}
\end{figure}

The hook partitions of $m$  include the single row $(m-1|0)$ and the single column $(0|m-1)$. The power sums can be written, by the Frobenius character formula, as $P_m=\sum _{a+b=m-1}(-1)^bQ_{(a|b)}$ (\cite{FultonHarris}, 4.10, 4.16). In particular,  $P_m$ is not affected by the mirror map. Since $\overline{[m]}=[m]$, we have also that $\overline{X}_m=X_m$.

The Pieri formula for products of Schur functions (\cite{Macdonald}, page 73) shows that $h_ie_j$ is the sum of the Schur functions of two hook partitions $(i-1|j)$ and $(i|j-1)$. We can write explicitly $h_i e_j=Q_{(i-1|j)}+Q_{(i|j-1)}$ for all $i,j\ge0$, by setting $Q_{(i|j)}=0$ when $i<0$ or $j<0$. 
We can use equation (\ref{eqnAH}) to write $A_m$ in the basis $\{Q_\gl\}_{\gl\vdash m}$, where the only partitions $\gl$ required are hooks of length $m$.

\begin{theorem}\label{Ahook}
\[A_m=\sum_{a+b=m-1}(-1)^bs^{a-b}Q_{(a|b)}.\]
\end{theorem}

\begin{proof}
By equation (\ref{eqnAH}) we have
\[A(t)=\frac{H(st)}{H(s^{-1}t)}={H(st)}{E(-s^{-1}t)}.
\]

Comparing the coefficients of $t^m$, taking $h_0=e_0=1$, gives
\begin{eqnarray*}zA_m&=&\sum_{i+j=m}s^ih_i(-1)^js^{-j}e_j=\sum_{i+j=m}(-1)^js^{i-j}h_i e_j\\
&=&\sum_{i+j=m}(-1)^js^{i-j}\left(Q_{(i-1|j)}+Q_{(i|j-1)}\right).
\end{eqnarray*} 

We can rewrite the sum as \begin{eqnarray*}&&\sum_{a+b=m-1} (-1)^{b}s^{a+1-b}Q_{(a|b)}+\sum_{a+b=m-1} (-1)^{b+1}s^{a-b-1}Q_{(a|b)}\\&=&(s-s^{-1})\sum_{a+b=m-1} (-1)^{b}s^{a-b}Q_{(a|b)},\end{eqnarray*} giving the result, since $z=s-s^{-1}$.
\end{proof}

The formula obtained in theorem \ref{Ahook} resembles the formula obtained by Rosso and Jones in  \cite{Rosso}, theorem~8. They remark there that the only partitions which occur when calculating the Homfly polynomial of the torus knots are hooks.  We show later how to deduce theorem \ref{Ahook} from theorem \ref{skeinplethysm}, which makes use of quantum invariants and \cite{Rosso} in its proof.

We can now give a simpler diagrammatic representation of $P_m$ using   theorem~\ref{Ahook} and the meridian map $\varphi$. For $Q\in\CC$ set $\Delta_\varphi (Q)=\varphi(Q)-\delta Q \in\CC$  and $\Delta_{\overline{\varphi}} (Q)=\overline{\varphi}(Q)-\delta Q 
\in\CC$.  
%Addendum
\begin{theorem}\label{Xdelta} 
We have 
\[(s-s^{-1})X_m= v \Delta _{\varphi }(\overline{A}_m) =-v^{-1}\Delta_{\overline{\varphi}}(A_m),\]
hence 
\[\{m\}P_m=v\Delta _{\varphi }(\overline{A}_m)=-v^{-1}\Delta_{\overline{\varphi}}(A_m).\]
\end{theorem} 
\begin{proof} Applying the mirror map to the 
equation of theorem \ref{Ahook} we get  
\[
\overline{A}_m=\sum_{(a|b)\vdash m}(-1)^bs^{b-a}Q_{(a|b)}.
\]

It follows that
\[
\Delta _{\varphi }(\overline{A}_m) =\sum_{(a|b)\vdash m}(-1)^bs^{b-a}\Delta _{\varphi }(Q_{(a|b)}).
\]

In general,
\be
\Delta _{\varphi }(Q_{\lambda})=v^{-1}(s-s^{-1})\sum_{x \in \lambda }s^{2c(x)}Q_{\lambda } \label{eqndeltaphi}
\ee
where the sum runs over cells $x\in\gl$  and $c(x)=j-i$ is the {\em content} of the cell $x$ in position $(i,j)$, which can be deduced from \cite{HadjiMorton}, theorem 3.4.

For $\lambda =(a|b)\vdash m$ we have $\ds\sum_{x \in \lambda }s^{2c(x)}=\frac{s^m-s^{-m}}{s-s^{-1}}s^{a-b}$, hence in particular
\[
\Delta _{\varphi }(Q_{(a|b)})=v^{-1}(s^m-s^{-m})s^{a-b}Q_{(a|b)}.
\]
Then
\begin{eqnarray*}
v\Delta _{\varphi }(\overline{A}_m) 
& = & (s^m-s^{-m}) \sum_{(a|b)\vdash m}(-1)^bQ_{(a|b)} \\
& = & (s^m-s^{-m}) P_m = \{m\}P_m\\
& = &  (s-s^{-1})X_m.
\end{eqnarray*}

Applying the mirror map gives the other representation.
\end{proof}

Hence we have an even simpler diagrammatic representative for $P_m$ in $\CC$ in terms of just two closed tangles,  as seen in figure \ref{figmeridianAm}.
\begin{figure}[ht!]
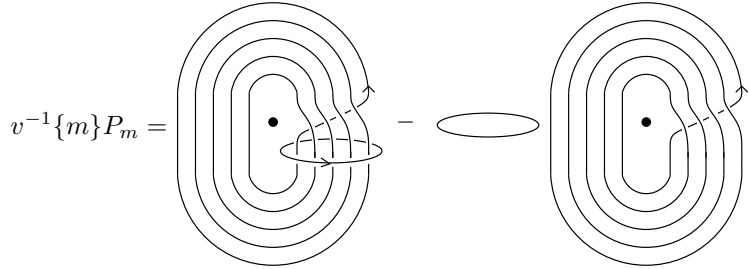

\labellist
\small
\pinlabel {$-$} at 346 217
\endlabellist
\bc
$v^{-1}\{ m\} P_m =$ \OtherRepresentationPm
\ec
\caption{A representation of $P_m$ by two tangles} \label{figmeridianAm}
\end{figure}

\section{Satellite and quantum $sl(N)$ invariants}\label{quantum}

One of the most useful features of the skein $\CC$ is its role in parametrising Homfly satellite invariants of a framed knot $K$. 

\subsection{Satellites}
 A satellite of $K$ is determined by choosing a diagram $Q$ in the standard annulus, and then drawing $Q$ on the annular neighbourhood of $K$ determined by the framing, to give the satellite knot $K*Q$. We refer to this construction as {\em decorating $K$ with the pattern $Q$} (see  figure \ref{figsatellite}). 

\begin{figure}[ht]
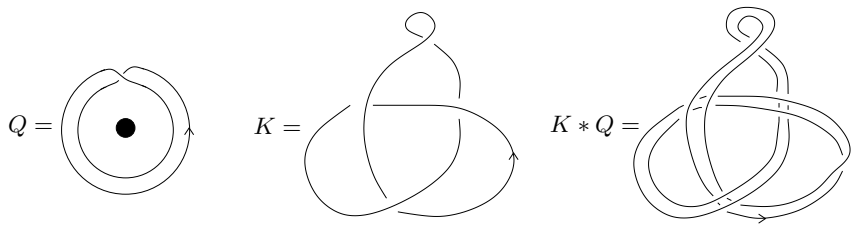

\labellist
\small
\pinlabel {$Q=$} at -60 195
\pinlabel {$K=$} at 430 195
\pinlabel {$K*Q=$} at 1060 195
\endlabellist
\bc
\Satellite
\ec
\caption{Satellite construction} \label{figsatellite}
\end{figure}

The Homfly polynomial $P(K*Q)$ of the satellite depends on $Q$ only as an element of the skein $\CC$ of the annulus, hence we  can extend the definition of  $K*Q$ to cover a general element $Q\in\CC$ if we are only concerned with its Homfly polynomial. We  regard $\CC$ as the natural parameter space for these invariants of $K$, known collectively as the {\em Homfly satellite invariants} of $K$. We use the notation $P(K;Q)$ in place of $P(K*Q)$ when we want to emphasise the dependence on $K$.  When $Q$ is restricted to lie in $\CC_m$ the invariants are called the $m$-string satellite invariants, and can be realised as linear combination of a finite number of satellite invariants. For example any closed $m$-braid $Q=\hat \beta$ in the annulus determines an element of $\CC_m$ which can be written as $Q=\sum_{\gl\vdash m} c_\gl Q_\gl$ in terms of the basis $\{Q_\gl\}_{ \gl\vdash m}$, with coefficients $c_\gl\in\Lambda$. The Homfly polynomial of the satellite $K*Q$ is then $P(K*Q)=\sum c_\gl P(K*Q_\gl)$.  

 The same overall collection of invariants of $K$ can be constructed from  the quantum invariants arising from the quantum groups $sl(N)_q$.

Here is a brief summary of the interrelations. A more extensive account can be found in the thesis of Lukac (\cite{Lukacthesis}), including details of   variant Homfly skeins with a framing correction factor, $x$. These are isomorphic to the skeins used here but the parameter allows a careful adjustment of the quadratic skein relation to agree directly with the natural relation arising from use of the quantum groups $sl(N)_q$. 

\subsection{Quantum invariants}

Quantum groups give rise to 1-parameter invariants $J(K;W)$ of  an oriented framed knot $K$ depending on a choice of finite dimensional module $W$ over the quantum group, following constructions of Turaev and others (\cite{{Turaev},{Wenzl},{AistonMorton}}). This choice is referred to as {\em colouring} $K$ by $W$, and can be extended for a link allowing a choice of colour for each component.

Fix a natural number $N$. When we colour $K$ by a finite dimensional module $W$ over the quantum group $sl(N)_q$, its invariant $J(K;W)$ depends  on one variable {$s$}. The invariant $J$ is linear under the direct sum of modules and all the modules over $sl(N)_q$ are semi-simple, so we can restrict our attention to the  irreducible modules $V_\gl^{(N)}$. For $sl(N)_q$ these are indexed by partitions $\gl$ with at most $N$ parts, without distinguishing two partitions which differ in some initial columns with $N$ cells.

 \begin{remark}(Comparison theorem)
\par
\begin{enumerate}
\item 
The  $sl(N)_q$ invariant for the irreducible module $V_\gl^{(N)}$ is the Homfly invariant for the knot decorated by $Q_\gl$ with $v=s^{-N}$, suitably normalised as in~\cite{Lukacthesis}. Explicitly,
\[P(K*Q_\gl)|_{v=s^{-N}}=x^{k|\gl|^2}J(K;V_\gl^{(N)})\] where $k$ is the writhe of $K$, $x=s^{1/N}$  and $|\gl|=\sum \gl _i$.

\item Each invariant $P(K*Q)|_{{v=s^{-N}}}$ is a linear combination $\sum c_\alpha J(K;W_\alpha)$ of quantum invariants .
 
\item Each $J(K;W)$ is a linear combination $\sum d_j P(K*Q_j)|_{{v=s^{-N}}}$ of  Homfly invariants.
 
\end{enumerate}
 
\end{remark}

\begin{remark} 
The 2-variable invariant $P(K*Q)$ can be recovered from the specialisations $P(K*Q)|_{{v=s^{-N}}}$ for sufficiently many $N$.
\end{remark}
\begin{remark}
If the pattern $Q$ is a closed braid on $m$ strings then we only need use partitions $\gl \vdash m$, since $\CC_m$ is spanned by $\{Q_\gl\}_{ \gl\vdash m}$.
Conversely, to realise $J(K;V_\gl^{(N)})$ with $ \gl\vdash m$  we can use closed $m$-braid patterns.
\end{remark}

\subsection{Basic constructions}
A quantum group $\cal G$ is an algebra over a formal power series ring ${\bf Q}[[h]]$, typically a deformed version of a classical Lie algebra. We write $q=e^h, s=e^{h/2}$ when working in $sl(N)_q$.
A finite dimensional module over $\cal G$ is a linear space on which $\cal G$ acts.

Crucially, $\cal G$ has a coproduct $\Delta$ which ensures that the tensor product ${V\otimes W}$ of two modules is also a module.
It also has a {\em universal $R$-matrix} (in a completion of ${\cal G}\otimes{\cal G}$) which determines a well-behaved module isomorphism \[R_{VW}:V\otimes W \to W\otimes V.\]

This has a diagrammatic view
indicating its use in converting coloured tangles to module homomorphisms:
\bc
{\labellist
\small
\pinlabel{{$ W\ \otimes \ V$}} at 165 156
\pinlabel{{$V\ \otimes \ W$}} at 165 12
\pinlabel{$R_{VW}$} at -50 84
\endlabellist}
\Rmatrix
\ec

A braid $\beta$ on $m$ strings with permutation $\pi\in S_m$ and a colouring of the strings by modules $V_1,\ldots,V_m$ leads to a module homomorphism 
\[J_\beta:V_1\otimes\cdots\otimes V_m \to V_{\pi(1)}\otimes\cdots\otimes V_{\pi(m)}\] using $R_{V_i,V_j}^{\pm1}$ at each elementary braid crossing.
The homomorphism $J_\beta$ depends {\em only on the braid} $\beta$ itself, not its decomposition into crossings, by the Yang-Baxter relation for the universal $R$-matrix.

When $V_i=V$ for all $i$ we get a module homomorphism $J_\beta:W\to W$, where 
$W=V^{\otimes m}$. Now any module $W$ decomposes as a  direct sum $\bigoplus {(W_\mu\otimes V_\mu^{(N)})}$, where $W_\mu \subset W$ is a linear subspace consisting of the {\em highest weight vectors} of type $\mu$  associated to the module $V_\mu^{(N)}$.   Highest weight subspaces of each type are preserved by module homomorphisms, and so $J_\beta$ determines (and is determined by) the restrictions $J_\beta(\mu):W_\mu \to W_\mu$ for each $\mu$, where $\mu $ runs over partitions with at most $N$ parts.

If a knot  (or one component of a link) $K$ is decorated by a pattern $T$ which is the closure of an $m$-braid $\beta$, then its quantum invariant $J(K*T;V)$ can be found from the endomorphism $J_\beta$ of $W=V^{\otimes m}$ in terms of the quantum invariants of $K$ and the restriction maps $J_\beta(\mu):W_\mu \to W_\mu$   by the formula
\be J(K*T;V)=\sum c_\mu J(K;V_\mu^{(N)}) \label{weighttrace}\ee with $c_\mu=\mbox{tr}\, J_\beta(\mu)$. This formula follows from lemma II.4.4 in Turaev's book \cite{Turaevbook}. We set $c_\mu=0$ when $W$ has no highest weight vectors of type $\mu$.

\subsection{Invariants of satellites}

The quantum invariant $J(K*T;V)$,  where $V=\bigoplus V_\gl^{(N)}$ is decomposed into irreducible modules, is the sum $\sum J(K*T;V_\gl^{(N)})$. This is given by  the sum of   Homfly satellite invariants $\sum P(K*T;Q_\gl)$, with $v=s^{-N}$, after adjustment by the framing correction parameter $x$. 

To discuss these further we note that the satellite $(K*T)*Q$ of $K*T$ when decorated with a pattern $Q$ can also be viewed as $K*(T*Q)$, namely the satellite of $K$ when decorated by the pattern $T*Q$ in the annulus.
 For a general element $Q=\sum c_i Q_i$ in $\CC$, written as a linear combination of diagrams $Q_i$, we can define $T*Q$ as an element of $\CC$ by $T*Q=\sum c_i T*Q_i$. This leads to the equation 
\[P(K*T;Q)=P(K;T*Q),\] where $T$ is a diagram in the annulus and $Q\in\CC$.

Hence we can find the Homfly polynomial $P(K*T;Q_\gl)$ as the satellite invariant $P(K;T*Q_\gl)$, which in turn can be found by writing   $T*Q_\gl$ in terms of the basis elements of the skein $\CC$.  Where $T$ is a closed $m$-braid and $\gl\vdash d$, this element lies in $\CC_{md}$ and we have
\[T*Q_\gl=\sum_{\mu\vdash md} a_\mu Q_\mu \] for some $a_\mu\in\Lambda$, giving
\[P(K*T;Q_\gl)=\sum_{\mu\vdash md} a_\mu P(K; Q_\mu) .\]

\begin{remark}The same is true if the diagram  $T$ is the closure of an $m$-tangle with all strings oriented in the same direction, but we must use the full basis elements $Q_{\nu,\rho}$ when $T$ is the closure of a tangle  with some reverse oriented strings.
\end{remark}

\section{Cables and plethysms}\label{cable}
The work of Rosso and Jones on traces in quantum groups, \cite{Rosso}, gives us a skein theoretic description of $T*Q$ in the annulus where $T=T_m^n$ is a {\em cable diagram}, and $Q$ is an element of $\CC_+$.

By the cable diagram $T_m^n$ we mean the diagram in the annulus formed by
closing the framed $m$-braid $(\beta_m)^n$ shown in figure \ref{torusbraid}.
With the blackboard framing, $T_m^n$ is the diagram of the $(m,n)$ torus
link with framing given by its neighbourhood in the surface of the torus. When $n=1$ we have $T_m^1=v^{-1}A_m$ as an element of $\CC$ because of the choice of framing.
\begin{figure}[ht]
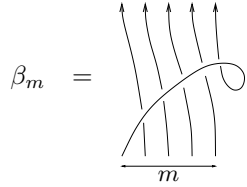

\labellist
\small
\pinlabel {$m$} at 134 -25
\endlabellist
\bc
$\beta_m$\quad=\quad \betam
\ec

\caption{$T_m^n$ is the closure of $(\beta_m)^n$}\label{torusbraid}
\end{figure}

If $m$ and $n$ have highest common factor $d$ we can regard $T_m^n$ as the $d$-fold parallel of a torus {\em knot} diagram, and reduce our calculations to the case where $m$ and $n$ are coprime.
In this case the 
  cable diagram $T_m^n$  induces a map $F_m^n:\CC\to\CC$ taking an element $Q\in \CC$ to the satellite $T_m^n*Q$.

The framing change map is the map $\tau=F_1^1$, illustrated in figure \ref{framingmap} by its effect on the 2-parallel element $(A_1)^2$. 

\begin{figure}[ht!]
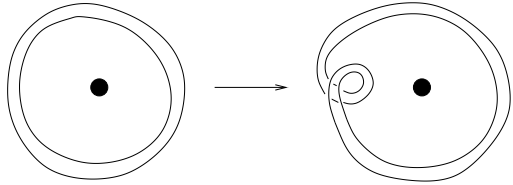

\labellist
\endlabellist
\bc
\framingmap
\ec
\caption{The framing change map on a $2$-parallel}\label{framingmap}
\end{figure}

It is shown in \cite{AistonMorton}, theorem 17, that $\tau (Q_{\lambda })=\tau _{\lambda  }Q_{\lambda  }$ where $\tau _{\lambda}=v^{k}s^r$ with $k=-|\gl|$ and $\ds r=2\sum_{x\in\gl} c(x)$.

We define a {\em fractional twist} map   $\tau ^{\frac{n}{m}}:{\cal C_+} \to {\cal C_+}$ as the linear map defined on the basis $\{Q_\gl\}$ by
\[
\tau ^{\frac{n}{m}}(Q_{\lambda  })=(\tau _{\lambda })^{\frac{n}{m}}Q_{\lambda }.
\]

\begin{remark} Since the basis vectors $\qlm$ for $\CC$ are also eigenvectors of the framing change map, \cite{HadjiMorton}, we could define $\tau ^{\frac{n}{m}}$ on the whole of $\CC$ in a similar way, using the fact that $\qlm$ has eigenvalue $\tau_{\gl,\mu}=v^ks^r$ with $k=|\mu|-|\gl|$ and $\ds r=2\sum_{x\in\gl} c(x)-2\sum_{x\in\mu} c(x)$.
\end{remark}

To give the formula for $F_m^n(Q)$ with $Q\in \CC_+$ we shall use the interpretation of $\CC_+$ as the ring of symmetric functions, and adopt the terminology of {\em plethysms} to describe the resulting elements of the skein.

\subsection{Plethysm} Let $p(x_1, \dots , x_N)=M_1+\dots +M_r$ be a symmetric polynomial in $N$ variables, which can be written as a sum of $r$ monomials, each with coefficient $1$. These include the Schur functions and the power sums.  Let $q(x_1, \dots , x_r)$ be a symmetric function in $r$ variables. The plethysm $q [p]$  is the symmetric function of $N$ variables
\[
q[ p] =q(M_1, \dots , M_r).
\]
\begin{remark}
A more general definition covering all symmetric polynomials $p$, along with further properties of plethysms, can be found in \cite{Macdonald}, where
 the notation $q\circ p$ is used in place of $q[p]$. We adopt here the notation from \cite{Remmel}.
\end{remark}

 We can write the symmetric polynomial $q[p]$ in the basis   of Schur functions as the  linear combination 
\[q [ p]=\sum _{\nu }b_{qp}^{\nu }s_{\nu }.\]
Determining the coefficients $b_{qp}^\nu$ is in general a non-trivial problem. If $p$ and $q$ are themselves Schur functions $s_{\lambda }$ and $s_{\mu}$ respectively, we write
\[s_{\mu} [s_{\lambda }]=\sum _{\nu }a_{\mu \lambda}^{\nu }s_{\nu }.\] It is shown in \cite{Macdonald} that $a_{\mu \lambda}^{\nu }$ is a non-negative integer in all cases. It is a feature of many such calculations with symmetric polynomials that the coefficients are independent of the number of variables, $N$, under the condition that we take $s_\nu=0$ when $\nu$ has more than $N$ parts.

Here are some properties of plethysms that we will use shortly:

\begin{enumerate}
    \item $q[p]$ is linear in $q$: $(a_1q_1+a_2q_2)[p] =a_1q_1 [p] + a_2q_2 [p]$, for any scalars $a_1,a_2$.
    \item In general $q[p]$ is not linear in $p$, but  if $q=P_m$ is a power sum, then $ q[A+B]=q[A]+q[B]$, where $A$ and $B$ are both sums of monomials. For if $A=A_1+ \dots + A_r$ and $B=B_1+ \dots + B_s$, then  
$P_m[A+B]=A_1^m+\dots +A_r^m+B_1^m+\dots +B_s^m=P_m [ A]+P_m [ B].$
    \item $s_{\lambda }[P_m]=P_m [s_{\lambda }]$ for any partition $\lambda$.
      \item $P_d [P_m ]= P_{md}= P_m [P_d]$, since
    $P_d [P_m] = (x_1^m)^d+ \dots + (x_r^m)^d=x_1^{md}+ \dots + x_r^{md}=P_{md}$.

\end{enumerate}    

 Let $Q$ be any element of $ \CC_+$, regarded as a symmetric function, and let $P\in\CC_+$ represent a sum of monomials each with coefficient $+1$, for example $P=P_m$ or $P=Q_\gl$.
We adopt the notation $Q[P]\in \CC_+$ for the element corresponding to the plethysm of the functions represented by $Q$ and $P$.

With this notation we can give a compact formula for the element $F_m^n(Q)=T_m^n*Q$ given by decorating the $(m,n)$ torus link in the annulus by $Q\in\CC_+$, shown schematically in figure \ref{Fmn}.

\begin{figure}[ht]
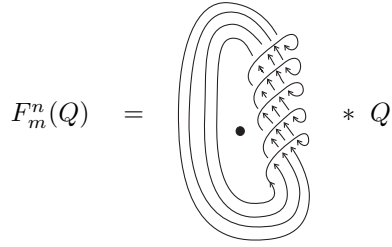

\bc
$F_m^n(Q) \quad = \quad \torusknot \quad *\ Q$
\ec
\caption{The element $F_m^n(Q)$ with $m=4, n=5$}\label{Fmn}
\end{figure}

This formula is purely in terms of the Homfly skein of the annulus, although the proof makes use of the formulae in \cite{Rosso} for quantum invariants of cables.

\begin{theorem} \label{skeinplethysm} Let $Q\in\CC_+$. Then \[F_m^n(Q)=\tau ^{\frac{n}{m}}(Q[ P_m]).\]
\end{theorem}
\begin{proof} Since $F_m^n, \tau ^{\frac{n}{m}}$ and the plethysm are linear in $Q$, it is enough to prove the result when $Q=Q_\gl$.

We start with an expression for $\tau ^{\frac{n}{m}}(Q_\gl[ P_m])$.
Recall (see section \ref{hook}) that $P_m=\sum _{\nu \vdash m } \omega _{\nu }Q_{\nu }$ with
\[
\omega _{\nu }=
\left\{
\begin{array}{cl}
0 & \mbox{if $\nu$  is not a hook partition},  \\
(-1)^b & \mbox{if $\nu $ is the hook partition $(a|b)$ of $m$}.
\end{array}
\right.
\]

Let $\gl\vdash d$.
Then $Q_\gl[P_m]=P_m[Q_\gl]=\sum_{\nu\vdash m} \omega_\nu Q_\nu[Q_\gl]$, so 
\begin{eqnarray*}\tau ^{\frac{n}{m}}(Q_\gl[P_m])&=&\tau ^{\frac{n}{m}}(\sum_{\nu\vdash m} \omega_\nu \sum_{\mu\vdash md}a_{\nu,\gl}^\mu Q_\mu)\\
&=&\sum_{\mu\vdash md} a_\mu Q_\mu,
\end{eqnarray*} with $a_\mu= (\tau_\mu)^{n/m}\sum_{\nu\vdash m} \omega_\nu a_{\nu,\gl}^\mu$.

We must now show that decorating the cable pattern $T_m^n$ with $Q_\gl$ gives the same result, in other words $F_m^n(Q_\gl)=T_m^n*Q_\gl=\sum_{\mu\vdash md}a_\mu Q_\mu$. 
It is enough to show that $P(K;T^n_m*Q_\gl)=\sum a_\mu P(K;Q_\mu)$ for all choices of knot $K$, and in turn it is enough to know this for all evaluations with $v=s^{-N}$.
 
Let $k$ be the writhe of $K$, hence the writhe of  $K*T_m^n$ is $m^2k+mn$. The comparison theorem establishes that
\[P(K;T^n_m*Q_\gl)=P(K*T_m^n;Q_\gl)=x^{d^2(m^2k+mn)}J(K*T_m^n;V_\gl^{(N)})\]
after the substitutions $v=s^{-N}$ and $x=s^{\frac{1}{N}}$.
 
We draw on \cite{Rosso} in calculating the invariant of cables coloured by a quantum group module to find $J(K*T_m^n;V_\gl^{(N)})$, where $N\geq l(\lambda )$.

The diagram $T_m^n$ is the closure of the framed braid $(\beta_m)^n$, which defines an endomorphism
 $J_{(\beta_m)^n}$ of $W=V^{\otimes m}$ when the braid $(\beta_m)^n$ is coloured by a module $V$ over $sl(N)_q$. We have $J(K*T_m^n;V)=\sum c_\mu J(K;V_\mu^{(N)})$ where $c_\mu$ is the trace of  $J_{(\beta_m)^n}$ restricted to the highest weight subspace $W_\mu\subset W$ of type $\mu$, using equation (\ref{weighttrace}).

Rosso and Jones  calculate the trace of such a restriction for a general quantum group. In their paper the subspace $W_\mu$ is called $M_\nu$, the braid is $\overline X_n^m$, with a slightly different framing, and the roles of $m$ and $n$ are interchanged. In our terminology they observe that $(\beta_m)^m$ operates as a scalar $f_\mu$ on $W_\mu$, and that consequently $J_{(\beta_m)^n}$ operates as $(f_\mu)^{n/m}$ times a matrix $D_\mu$ with integer trace. Our choice of framing on $\beta_m$ ensures that the scalar $f_\mu$ does not depend on $V$.

 The trace of $D_\mu$ is independent of the quantum parameter and can be calculated from classical invariant theory in terms of the decomposition of $V^{\otimes m}$ as $\bigoplus_{\nu\vdash m}(N_\nu \otimes L_\nu)$, where the symmetric group $S_m$ acts on $N_\nu$ and $gl(V)$ acts on $L_\nu$ by the irreducible representation given by the partition $\nu$. (Rosso and Jones use $Y$ in place of $\nu$ here, and $\nu$ in place of $\mu$). They decompose $L_\nu$ further into irreducible modules $V_\mu$ over the classical version of the quantum group, with some multiplicity $[\nu:\mu]$, and derive the  formula   $\mbox{tr}\, D_\mu=\sum_{\nu\vdash m}\chi_\nu(\sigma)[\nu:\mu]$, where $\chi_\nu$ is the character defined by the representation of $S_m$ on $N_\nu$, and $\sigma$ is the permutation of the braid $(\beta_m)^n$.

When $m$ and $n$ are coprime the permutation is an $m$-cycle, and its character is $\chi_\nu(\sigma)=\omega_\nu$ in the terminology above. 

Where $V=V_\gl^{(N)}$ is the irreducible representation of $sl(N)_q$ of highest weight $\gl\vdash d$, the decomposition of $L_\nu$ is given by the plethysm coefficients  $a_{\nu,\gl}^\mu $ for $s_\mu[s_\gl]$, following the interpretation of plethysms as composite of representations of general linear or symmetric groups \cite{{Macdonald},{Remmel}}. This gives $L_\nu=\sum_{\mu \vdash md}a_{\nu,\gl}^\mu V_\mu^{(N)}$, leading to the formula
\[\mbox{tr}\, D_\mu=\sum_{\nu\vdash m}\omega_\nu a_{\nu,\gl}^\mu.\]
 
It follows that $c_\mu = (f_\mu)^{n/m}\mbox{tr}\, D_\mu  =   (f_\mu)^{n/m}\sum_{\nu\vdash m}\omega_\nu
a_{\nu,\gl}^\mu $, hence
\[
J(K*T_m^n;V_\gl^{(N)})
=\sum_{\mu\vdash md} c_\mu J(K;V_\mu^{(N)})
=\sum_{\mu\vdash md} (f_\mu)^{n/m}\sum_{\nu\vdash m}\omega_\nu a_{\nu,\gl}^\mu J(K;V_\mu^{(N)}).
\]

Then
\[
P(K;T^n_m*Q_\gl)=x^{d^2(m^2k+mn)}\sum_{\mu\vdash md} c_\mu J(K;V_\mu^{(N)})
\]
  while $\sum a_\mu P(K;Q_\mu)=x^{(md)^2k}\sum a_\mu J(K;V_\mu^{(N)})$, also by the comparison theorem.
These will give the same value provided that $x^{d^2mn}c_\mu=a_\mu$, and for this it is enough to know that $x^{(md)^2}f_\mu=\tau_\mu$, when $v=s^{-N}$ and $x=s^{1/N}$.
 
Now theorem \ref{skeinplethysm} holds when $m=n=1$, since $T_1^1*Q_\gl=\tau_\gl Q_\gl$ for any $\gl$, by the definition of $\tau_\gl$. Taking $V=V_\gl^{(N)}$ with $m=1$ gives $W=V_\gl^{(N)}$. This has only one non-zero highest weight subspace $W_\gl$, which has dimension $1$, and the braid $(\beta_1)^1$ acts on it by the scalar $f_\gl$, so $J(K*T_1^1;V_\gl^{(N)})=f_\gl J(K;V_\gl^{(N)})$. The comparison theorem   shows on the one hand that 
\[P(K*T_1^1;Q_\gl)=P(K;T_1^1*Q_\gl)=\tau_\gl P(K;Q_\gl)=x^{k|\gl|^2}\tau_\gl J(K; V_\gl^{(N)}),\] and on the other hand that \[
P(K*T_1^1;Q_\gl)= x^ {(k+1)|\gl|^2}f_\gl J(K;V_\gl^{(N)})\]
 for any $K$, with $v=s^{-N}$. Taking $K$ to be the trivial knot then establishes the relation $\tau_\gl=x^{|\gl|^2}f_\gl$ for all $\gl$, with $v=s^{-N}$ and $x=s^{1/N}$, which completes the proof.
\end{proof}

  As a corollary to theorem \ref{skeinplethysm} we have the following formula for  $T_m^n*P_d$, entirely in the Homfly skein $\CC$,  when the $(m,n)$ cable pattern $T_m^n$ is decorated by a power sum:
\begin{corollary}\label{corcable}
\begin{equation}\label{eqncable} T_m^n*P_d=F_m^n(P_d)=\tau^{\frac{n}{m}}(P_{md})=\sum _{\nu \vdash md}(\tau_{\nu})^{\frac{n}{m}}\omega _{\nu }Q_{\nu }.
\end{equation}
\end{corollary}
Note that only $md$-hooks $\nu$ appear in this formula, because $\omega_\nu=0$ for any other partition $\nu\vdash md$.

Theorem \ref{Ahook} is the special case of (\ref{eqncable}) when $n=d=1$. Indeed, we
have \[A_m=vT_m^1=v\sum \omega_\nu (\tau_\nu)^{\frac{1}{m}}Q_\nu=\sum_{a+b+1=m}(-1)^b s^{a-b}Q_{(a|b)},\] since the sum of the contents  of an $m$-hook $(a|b)$ is $\frac{1}{2}m(a-b)$ and so $\tau_{(a|b)}=v^{-m}s^{m(a-b)}$.

\section{Decorating by power sums}\label{powercable}

We now use corollary \ref{corcable} to establish some results   in the Homfly skein, where we consider decorations of knots or links by power sums $P_M$. In part this has been encouraged by the known and conjectured integrality results  for such invariants from \cite{LM} and \cite{OV} and work by Garoufalidis and Le in trying to develop some direct skein theoretic properties of these invariants.  In particular, substitution of the element $X_M$ in place of $P_M$ can have a good effect because this can be represented by a positive integer sum of diagrams, and then the integrality of the standard Homfly polynomial  for links can be used.

We don't have a general means of working purely with power sum decorations at the skein level of the underlying diagrams, but in theorem \ref{powerHopf} we give   a relation in $\CC$ between some diagrams when decorated by power sums,  originally conjectured by  the first author  in the course of a visit to Garoufalidis at Georgia Institute of Technology in 2003.

\subsection{Murphy operators in the Hecke algebras}
 The Murphy operators $T_1,\ldots,T_n$ are commuting elements in the Hecke algebra $H_n$, where $T_i$ is represented by the framed diagram in figure \ref{Ti}. 
  The framing used here, which is inherited from the surface of a vertical cylinder, agrees with that used for $T_i$ in \cite{AistonMorton}, while the element $T(i)$ in \cite{skein} is $T_i$ with the blackboard framing as a braid.  Any symmetric polynomial $f(T_1,\ldots,T_n)$ is in the centre of $H_n$. More details of this description of $H_n$ and the Murphy operators can be found in \cite{skein}. 

\begin{figure}[ht!]
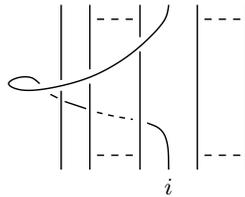

\labellist
\small
\pinlabel {$i$} at 203 -19
\endlabellist
\bc
\Tiblackboardframing
\ec

\caption{A diagram for $T_i$ with the blackboard framing}\label{Ti}
\end{figure}

For $\mu\vdash n$ write $e_\mu\in H_n$ for Aiston's idempotent \cite{AistonMorton}, whose closure in $\CC$ is $Q_\mu$.
The details of the following result are due to Lukac:
\begin{theorem}\label{symmetricMurphy} The symmetric polynomial $f(T_1,\ldots,T_n)$ satisfies 
\[f(T_1,\ldots,T_n)e_\mu=f(y_1,\ldots,y_n)e_\mu,\] where  $\{y_1,\ldots,y_n\}=\{v^{-1}s^{2c(x)}\}_{x\in\mu}$ as an unordered set.
\end{theorem}
\begin{proof} The result can be established, using theorem 17 of \cite{AistonMorton}, in the case where $f$ is an elementary symmetric function, choosing an ordering for the cells, and the corresponding Murphy operators, so that the first $k$ cells form a legitimate Young diagram for each $k$. 

Since any symmetric polynomial is a polynomial in the elementary symmetric functions the general case follows.
\end{proof}
 
\subsection{Power sums and cables}
For $Y\in\CC$ we define a linear map $\Delta_Y:\CC\to\CC$ by $\Delta_Y(Q)=\varphi_Y(Q)-P(Y) Q$. Here    $\varphi_Y(Q)\in\CC$ is the element illustrated in  figure \ref{figphiY} and $P(Y)$ is the Homfly polynomial of $Y$ regarded as a decoration of the unknot in the plane.

\begin{figure}[ht!]
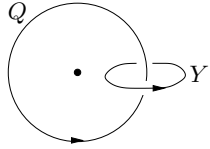

\labellist
\small
\pinlabel {$Q$} at 20 260
\pinlabel {$Y$} at 380 140
\endlabellist
\bc
\meridianmap
\ec
\caption{The element $\varphi_Y(Q) \in \CC$, where $Q,Y \in \CC$}\label{figphiY}
\end{figure}

When $Y$ is the core curve $A_1=P_1$ in the annulus, we have earlier used the notation $\Delta_\varphi$ in place of $\Delta_{P_1}$ in theorem \ref{Xdelta}.

We know that $\qlm$ is an eigenvector for any $\Delta_Y$, by \cite{HadjiMorton}, since $\Delta_Y$ commutes with $\varphi$ and  all
the eigenvalues of $\varphi$ have multiplicity 1.
Write $\Delta_\varphi(Q_\mu)=\Delta_{P_1}(Q_\mu)=t_\mu Q_\mu$, where   $t_\mu=\{1\}v^{-1}\sum_{x\in\mu}s^{2c(x)}$, by equation (\ref{eqndeltaphi}). 
We define $\psi_N:\Lambda\to \Lambda$ by $\psi_N(s)=s^N, \psi_N(v)=v^N$.

\begin{lemma}\label{DeltaPN} $\Delta_{P_N}(Q_\mu)= \psi_N(t_\mu)Q_\mu=\{N\}v^{-N}\sum_{x\in\mu}(s^N)^{2c(x)}Q_\mu$.
\end{lemma}
\begin{proof}  Suppose that $\mu$ is a partition of $M$.
The element $\Delta _{P_{N}}(Q_{\mu })$ is the closure of   $Y_{M,N}\,e_{\mu }$, where $Y_{M,N}$ is the following
element of the Hecke algebra $H_M$:
\[Y_{M,N}={\labellist
\small
\pinlabel {$P_N$} at 1 145
\pinlabel {$M$ strings} at 46 195
\endlabellist} \Heckemeridian \ -\ {\labellist
\small
\pinlabel {$P_N$} at 45 153
\pinlabel {$M$ strings} at 122 195
\endlabellist} \deltaIdentity\]

Theorem 3.9 in \cite{skein} establishes the equation $Y_{M,N}=\{ N\} \sum _{j=1}^MT_j^N$ in $H_M$,
where the framing for each $T_j$ is given above. It follows that $Y_{M,N}e_\mu=\{N\}\sum_{x\in\mu}(v^{-1}s^{2c(x)})^N e_\mu$ by theorem \ref{symmetricMurphy}.
Taking the closure gives \[\Delta_{P_N}(Q_\mu)=\{N\}\sum_{x\in\mu}(v^{-1}s^{2c(x)})^N Q_\mu.\]

When $N=1$ we have $\Delta_\varphi(Q_\mu)=\{1\}\sum_{x\in\mu}v^{-1}s^{2c(x)} Q_\mu =t_\mu Q_\mu$, while the general result reads $\Delta_{P_N}(Q_\mu)=\psi_N(t_\mu) Q_\mu$.
\end{proof}

\begin{theorem} \label{powerHopf} $\Delta_{P_N}(P_M)=\{MN\}T_m^n *P_d$, where $d=\gcd(M,N)$, $M=md$ and $ N=nd$.
\end{theorem}

\begin{proof}
\begin{eqnarray*}\Delta_{P_N}(P_M)&=&\sum_{\mu\vdash M}\omega_\mu \Delta_{P_N}( Q_\mu)\\
&=&\sum_{\mu\vdash M}\omega_\mu \psi_N(t_\mu) Q_\mu\\
&=&\sum_{\mu\vdash M}\omega_\mu (\{N\}\sum_{x\in\mu}v^{-N} (s^N)^{2c(x)}) Q_\mu\\
&=&\sum_{\mu=(a|b)\vdash M}\omega_\mu v^{-N} ((s^N)^M-(s^N)^{-M}) s^{N(a-b)}Q_\mu\\
&=&\sum_{\mu=(a|b)\vdash M}\omega_\mu v^{-N} \{MN\} s^{N(a-b)}Q_\mu .
\end{eqnarray*}

Now equation (\ref{eqncable}) shows that \[T_m^n*P_d=\sum_{\mu\vdash md}\omega_\mu (\tau_\mu)^{n/m}Q_\mu.\]

For a hook partition $\mu =(a|b)\vdash md=M$  we have $2\sum_{x\in\mu} c(x)=M(a-b)$, hence $\tau_{\mu}=v^{-M}s^{M(a-b)}$.
Then $(\tau_{(a|b)})^{n/m}=(\tau_{(a|b)})^{N/M}=v^{-N}s^{N(a-b)}$, giving the result.
\end{proof} 

We show theorem \ref{powerHopf} in diagrammatic form in figure \ref{figpowerHopf}, where $m=4, n=5$ in the torus knot diagram, and all the diagrams are decorated by power sums.
 
\begin{figure}[ht]
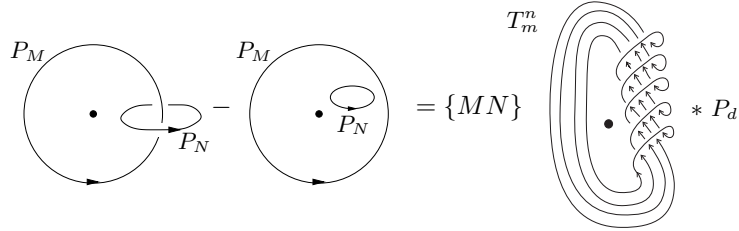

\bc
{\labellist
\small
\pinlabel {$P_M$} at 10 266
\pinlabel {$P_N$} at 340 85
\endlabellist}
\meridianmap $\ -\ ${\labellist
\small
\pinlabel {$P_M$} at 10 266
\pinlabel {$P_N$} at 207 120
\endlabellist}
\deltaA 
 \quad $=\{ MN \}\   $ {\labellist
\small
\pinlabel {$*\ P_d$} at 600 330
\pinlabel {$T_m^n$} at 220 506
\endlabellist}
\torusknot
\ec
\caption{In this formula $d=\gcd(M,N)$, $M=md$ and $N=nd$} \label{figpowerHopf}
\end{figure}

\section{Examples}\label{Ex}
We conclude with some explicit special cases of theorem \ref{powerHopf}, which inspired its general formulation.

\begin{example} When $M$ and $N$ are coprime then $d=1, m=M$, $n=N$, and so
 $\Delta_{P_n}(P_m)=\{mn\}T_m^n$, the framed $(m,n)$ torus link in the solid torus.
\end{example}

\begin{example} When $N$ is a multiple of $M$ then $d=M, m=1$, $N=Mn$, and so
   $\Delta_{P_{Mn}}(P_M)=\{M^2n\}T_1^n*P_M$.

In particular where $M=N$, of interest in considering links with all components decorated by $P_M$, we have $m=n=1$, $d=M$ and $\Delta_{P_{M}}(P_M)=\{M^2\}T_1^1*P_M$, giving a formula for the effect of the framing change on $P_M$ in terms of other links decorated by $P_M$ (see figure \ref{figPM}).
\end{example}

\begin{figure}[ht]
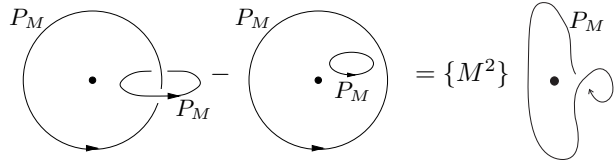

\bc
{\labellist
\small
\pinlabel {$P_M$} at 10 266
\pinlabel {$P_M$} at 340 85
\endlabellist}
\meridianmap $\ -\ ${\labellist
\small
\pinlabel {$P_M$} at 10 266
\pinlabel {$P_M$} at 207 120
\endlabellist}
\deltaA 
 \quad $=\{ M^2 \}   $ {\labellist
\small
\pinlabel {$ P_M$} at 290 450
\endlabellist}
\framechange
\ec
\caption{The effect of the framing change on $P_M$} \label{figPM}
\end{figure}

\begin{example} When $M$ is a multiple of $N$, then $d=N, n=1$, $M=mN$ and so
$\Delta_{P_{N}}(P_{mN})=\{mN^2\}T_m^1*P_N$.

The special case where $N=1$ gives $\Delta_\varphi(P_m)=\Delta_{P_{1}}(P_m)=\{m\}T_m^1=v^{-1}\{m\}A_m$.
\end{example}

This may be compared with the formula $\Delta_{\overline{\varphi}}(A_m) =-v\{m\}P_m$ obtained in theorem
\ref{Xdelta}, which leads to the formula
\[\Delta_{\overline{\varphi}}\Delta_\varphi(P_m)=-\{m\}^2 P_m.\]

Indeed, lemma \ref{DeltaPN} shows that $\Delta_{\overline{\varphi}}\Delta_\varphi$ operates as the scalar $-\{m\}^2$ on the subspace of $\CC_+$ spanned by $m$-hooks, for each $m$. More generally, $\Delta_{P_N^*}\Delta_{P_N}$ operates as $-\{mN\}^2$ on this subspace, where $P_N^*\in\CC$ is $P_N$ with the string orientations reversed.

\section*{Acknowledgments}  The first author  thanks Stavros Garoufalidis for hospitality and   support received at Georgia Institute of Technology in 2003 when the ideas in the later part of this paper were formulated. He is also grateful to Universidad Complutense, Madrid for support during work on the paper in 2007.  

The second author is grateful to The University of
Liverpool for its hospitality during work on this paper in the spring of 2005, financed by the Secretar\'{\i}a de Estado de
Universidades e Investigaci\'on del Ministerio de Educaci\'on y Ciencia, Spain, ref. PR2005-0099.

\end{document}